\documentclass{SCIYA2017enOL}
\online
\begin{document}

\ensubject{fdsfd}%¶þŒ¶Ñ§¿Æ

%%%%%%%%%%%%%%%%%%%%%%%%%%%%%%%%%%%%%%%%%%%%%%%%%%%%%%%
%%% Authors do not modify the information below
%%% ×÷Õß²»ÐèÒªÐÞžÄŽËŽŠÐÅÏ¢
%%% ÓÐ×šÌâÃû³ÆÊ±, œ«µÚÒ»ÐÐµÄ{}×¢ÊÍµô, Ê¹ÓÃµÚ¶þÐÐ
\ArticleType{ARTICLES}%ÀžÄ¿
%\SpecialTopic{Progress of Projects Supported by NSFC}%×šÌâ
%\SubTitle{Dedicated to Professor Aaaaa Bbbbb on the Occasion of his {\rm 80}th Birthday}%×š¿¯ËµÃ÷
\Year{2023}
\Month{January}%
\Vol{66}
\No{1}
\BeginPage{1} %
\DOI{10.1007/s11425-016-5135-4}
\ReceiveDate{February 8, 2023}
\AcceptDate{May 5, 2023}
\OnlineDate{  , 2023}
%%%%%%%%%%%%%%%%%%%%%%%%%%%%%%%%%%%%%%%%%%%%%%%%%%%%%%%
\newcommand{\Y}{\mathbin{\mathsf Y}})
%%% title: ±êÌâ
%%%   \title{title}{title for citation}
\title[]{The unitary subgroups of group algebras of a class of  finite $2$-groups with derived subgroup of order $2$ }
{The unitary subgroups of group algebras of a class of  finite $2$-groups with derived subgroup of order $2$ }%%ºó±ß»šÀšºÅÊÇÃŒÌâ

%%% Corresponding author: ÍšÐÅ×÷Õß
%%%   \author[number]{Full name}{{email@xxx.com}}
%%% General author: Ò»°ã×÷Õß
%%%   \author[number]{Full name}{}
\author[1]{Yulei Wang}{{yulwang@haut.edu.cn}}
\author[2,$\ast$]{Heguo Liu}{{ghliu@hainanu.edu.cn}}

%%% Author information for page head. Ò³ÃŒÖÐµÄ×÷ÕßÐÅÏ¢
%%% ÈôŽËŽŠÖž¶šÒÔŽËŽŠÎª×Œ, ·ñÔòÖ±œÓµ÷ÓÃauthorÐÅÏ¢
\AuthorMark{Yulei Wang}

%%% Authors for citation. Ê×Ò³ÒýÓÃÖÐµÄ×÷ÕßÐÅÏ¢
%%% ÈôŽËŽŠÖž¶šÒÔŽËŽŠÎª×Œ, ·ñÔòÖ±œÓµ÷ÓÃauthorÐÅÏ¢
\AuthorCitation{Yulei Wang, Heguo Liu}

%%% Address. µØÖ·
%%%   \address[number]{Address, City {\rm Postcode}, Country}
\address[1]{Department of Mathematics, Henan University of Technology,
Zhengzhou {\rm 450001}, China}
\address[2]{Department of Mathematics, Hainan University, Haikou {\rm 570228}, China}

%%% Abstract. ÕªÒª
\abstract{Let $p$ be a prime and $F$ be a finite field of characteristic $p$. Suppose that  $FG$ is the group algebra of the finite $p$-group $G$ over the field $F$. Let $V(FG)$ denote the  group of normalized units in $FG$
and let  $V_*(FG)$ denote the unitary subgroup  of $V(FG)$.
If $p$ is odd, then the order of $V_*(FG)$ is $|F|^{(|G|-1)/2}$.
However, the case when $p=2$ still is open.
In this paper, the order of $V_*(FG)$ is computed when
$G$ is  a nonabelian  $2$-group given by a
central extension of the form
$$1\longrightarrow \mathbb{Z}_{2^n}\times \mathbb{Z}_{2^m} \longrightarrow G \longrightarrow \mathbb{Z}_2\times
\cdots\times \mathbb{Z}_2 \longrightarrow 1$$ and $G'\cong
\mathbb{Z}_2$, $n, m\geq 1$.  Further, a conjecture is confirmed, namely,
the order of $V_*(FG)$ can be divisible by $|F|^{\frac{1}{2}(|G|+|\Omega_1(G)|)-1}$, where $\Omega_1(G)=\{g\in G\ |\ g^2=1\}$.}

%%% Keywords. ¹ØŒüŽÊ


\maketitle

%%%%%%%%%%%%%%%%%%%%%%%%%%%%%%%%%%%%%%%%%%%%%%%%%%%%%%%
%%% The main text. ÕýÎÄ²¿·Ö%
%%  ÍŒ±íÒýÓÃ\cref¹«ÊœÒýÓÃ\eqref²Î¿ŒÎÄÏ×\cite
%%%%%%%%%%%%%%%%%%%%%%%%%%%%%%%%%%%%%%%%%%%%%%%%%%%%%%%
\section{Introduction}
In this paper, $p$ always is a  prime, $F$ is a finite field of characteristic $p$ and $G$ is a finite $p$-group.

For an integral ring $Z$, let $U(Z)$  and  $U(ZG)$ be the multiplicative group of $Z$ and the integral group ring $ZG$, respectively.
 Suppose that $f$ is a homomorphism of the group $G$ into $U(Z)$, we define an anti-automorphism of
the ring $ZG$:
  $x^f=\sum\limits_{g\in G}\alpha_gf(g)g^{-1}\in ZG$, where $x=\sum\limits_{g\in G}\alpha_gg\in ZG$.
An element $u\in U(ZG)$ is said to be $f$-unitary if the inverse element $u^{-1}$ coincides with the element $u^f$ or $-u^f$. Obviously, all $f$-unitary elements of
the group $U(ZG)$ form a subgroup, which is denoted by $U_f(ZG)$. The interest in $U_f(ZG)$ arouse from algebraic topology and unitary $K$-theory in \cite{Artamonov, Novikov}.
The study and description about $U_f(ZG)$ in certain cases is known as Novikov's problem.

In particular, $f$ is trivial, namely,  $f(g)=1$ for all $g\in G$. At this time, we denote this trivial homomorphism $f$ by $*$.
Let $V(FG)$ be the group of normalized units in $FG$, namely
$$V(FG)=\left\{\sum\limits_{g\in G}\alpha_gg\in FG \ |\ \sum\limits_{g\in G}\alpha_g=1\right\}.$$
Obviously, $V(FG)$ is a subgroup of $U(FG)$. An element $x\in V(FG)$ is called unitary (also called unitary normalized unit) if $x^*=x^{-1}$.
We denote by $V_*(FG)$ the subgroup of all unitary elements of $V(FG)$.
The problem of the description of invariants of $V_*(FG)$ was raised by  Novikov,
and  Serre has showed that there is a relation between the self-dual normal basis of the finite Galois extension $L$ over $F$ with Galois group $G$ and the unitary subgroup of the group algebra $FG$ in \cite{Serre},
where $char(F)=2$.

The order of $V_*(FG)$ is equal to $|F|^{(|G|-1)/2}$ when $char(F)>2$ in \cite{BovdiR}. Thus we only consider how to compute
order of $V_*(FG)$ when $p=2$. In fact, it is particularly challenging to compute the order of $V_*(FG)$ when $p=2$ as Balogh put in \cite{Balogh}. At this time, there is an interesting conjecture in \cite{BovdiSa}, namely,  the order of $V_*(FG)$ is divisible by $F^{\frac{1}{2}(|G|+|\Omega_1(G)|)-1}$,
where $\Omega_1(G)=\{g\in G\ |\ g^2=1\}$.

For  finite abelian $p$-groups, the relative results of  $V_*(FG)$ have been obtained in \cite{BaloghL, BaloghB, BovdiSa, BovdiSz}.
However, when $G$ is nonabelian, a few facts about $V_*(FG)$ can be known.  For several special
 classes of groups, some results can only be obtained, see \cite{BaloghB, BaloghCG, CreedonG, Gildea, KaurK,WL}. In particular, the orders of
unitary subgroups of group algebras of a class of extraspecial $2$-groups, a dihedral group, a quaternion group have been determined in \cite{BovdiR}.
In \cite{Balogh},  Balogh gave the order of unitary subgroups of group algebras of finite $2$-groups which satisfies special conditions.
 Blackburn determined the isomorphism types of groups of prime power order with derived subgroup of prime order in \cite{Blackburn},
but we need further determine the more accurate structure in order to compute the unitary subgroups of group algebras of  the class of $2$-groups.
 In \cite{WL}, we studied the unitary normalized units of  a nonabelian  $2$-group with a derived subgroup of order $2$, which is given by a
central extension of a cyclic group by an elementary abelian $2$-group.
Now we will consider a nonabelian  $2$-group $G$ given by a
central extension of the form
$$1\longrightarrow \mathbb{Z}_{2^n}\times  \mathbb{Z}_{2^m}\longrightarrow G \longrightarrow \mathbb{Z}_2\times
\cdots\times \mathbb{Z}_2 \longrightarrow 1$$ and $G'\cong
\mathbb{Z}_2$, $m, n\geq 1$.

First, we will give some fundamental conclusions in the second section. And then we give the accurate
structure of the above group $G$ according to the inner abelian finite $2$-groups and central product in the third section. In the forth section, we will determine the unitary subgroups of the group algebra $FG$, where $char(F)=2$. Moreover, a conjecture about the order of $V_*(FG)$  is confirmed.

\section{Preliminaries}

To avoid confusion, we will explain some notations.
Let $G[2^i]$ denote the subgroup $\langle g\in G \ |\  g^{2^i}=1\rangle$ and let $G^{2^i}$ denote the subgroup $\langle g^{2^i} \ |\ g\in G\rangle$. Let $G^{(k)}$  denote the direct product of $k$ groups $G's$.
For $x=\sum\limits_{g\in G}\alpha_gg\in FG$, the support $\{g\in G\ |\ \alpha_g\neq 0\}$ of $x$ is denoted by $supp(x)$.
For any subset $S$ of $G$, we define $\widehat{S}:=\sum\limits_{g\in S}g$.
For an element $c$ of $G$,  we denote by $\Omega_c(G)$ the set $\{g\in G\ |\ g^2=c\}$.
We denote by $d(G)$ the number of the minimal generated elements of finite $p$-group $G$.
The combinatorial number is denoted by $\binom{n}{i}$
, namely, $\binom{n}{i}=\frac{n!}{(n-i)!i!}$.
Let $$\gamma_1(k):= \binom{k}{0}3^{k}+ \binom{k}{2}3^{k-2}+\cdots+ \binom{k}{l}3^{k-l}=2^{2k-1}+2^{k-1},$$ where $l=k$ if $k$ is even
and $l=k-1$ if $k$ is odd. Let $$\gamma_2(k):=  \binom{k}{1}3^{k-1}+ \binom{k}{3}3^{k-3}+\cdots+ \binom{k}{l}3^{k-l}=2^{2k-1}-2^{k-1},$$ where $l=k-1$ if $k$ is even
and $l=k$ if $k$ is odd.
Other notations used are standard (as in \cite{Knus, Robinson}).

\begin{definition}
A group $G$ is a central product of the normal subgroups $H$ and $K$ if $G=HK$, $[H,K]=1$, and  denoted
by $H\Y K$. The central product of $k$ groups $H's$ is denoted by  $H^{\Y k}$.
\end{definition}

\begin{definition}
A finite $p$-group $G$ is called inner abelian  if  $G$ is nonabelian and every proper subgroup of $G$ is abelian.
\end{definition}

\begin{lemma}[\cite{Bovdi Sa}] \label{AU}
 Let $G$ be a finite abelian $2$-group and $F$ a finite field of characteristic $2$. Then
$$|V_{*}(FG)|=|G^2[2]|\cdot |F|^{\frac{|G|+|\Omega_1(G)|}{2}-1}.$$
\end{lemma}

\begin{lemma}[\cite{Bovdi R}] \label{DQ}

 Let $F$ be a finite field of characteristic $2$.

 (i) If $G$ is a dihedral group of order $2^{n+1}$, then $|V_*(FG)|=|F|^{3\cdot 2^{n-1}}$.

 (ii) If $G$ is a quaternion group of order $2^{n+1}$, then $|V_*(FG)|=4|F|^{2^{n}}$.

\end{lemma}

\begin{lemma}[\cite {Balogh}] \label{CCU}
Let $F$ be a finite field of characteristic $2$. Let $G=K\times E$, where $K=\langle a,b\ |\ a^4=b^4=1, [a,b]=a^2 \rangle$
and $E$ be a finite elementary abelian $2$-group. Then $|V_{*}(FG)|=4 |F|^{\frac{|G|+|\Omega_1(G)|}{2}-1}$.

\end{lemma}

In \cite{KaurK},  Kaur and Khan studied the unit group $U(F(G\times A))$ and
unitary unit group  $U_{*}(F(G\times A))$ of  the group algebra $F(G\times A)$
of the direct product of an arbitrary finite group $G$ and a finite elementary abelian $2$-group $A$ over a field $F$ with characteristic $2$.
Similarly, we may obtain the results of $V_{*}(F(G\times A))$.

\begin{lemma} \label{UDP}

Let $F$ be a finite field of characteristic $2$. Suppose that $G$ is a finite $2$-group and  $A$ is an elementary abelian $2$-group of order $2^k$.
Then the unitary normalized unit subgroup $V_{*}(F(G\times A))$ is semidirect product of the group $W^*$
and the unitary normalized unit subgroup $V_*(FG)$,
where $W^*=(\cdots((A^*_k\rtimes A^*_{k-1})\rtimes A^*_{k-2}\rtimes\cdots)\rtimes A^*_1$ such that each $A^*_i$
is an elementary abelian $2$-group of order $|F|^{2^{i-2}(|G|+|\Omega_1(G)|)}$ and the order of
$V_{*}(F(G\times A))$ is $|V_{*}(FG)|\cdot|F|^{\frac{1}{2}(|G|+|\Omega_1(G)|)(|A|-1)}$.
Further, if $|V_{*}(FG)|=l|F|^{\frac{1}{2}(|G|+|\Omega_1(G)|)-1}$ for some nature number $l$, then
$|V_{*}(F(G\times A))|=l|F|^{\frac{1}{2}(|G\times A|+|\Omega_1(G\times A)|)-1}$.

\end{lemma}

\begin{proof}
Let $A=\langle a_1\rangle\times\langle a_2\rangle\times\cdots\times\langle a_k\rangle$.
Write $G_k:=G\times A=G_{k-1}\times \langle a_k\rangle$ and $G_0:=G$,
where $G_{k-1}=G\times\langle a_1\rangle\times\cdots\times\langle a_{k-1}\rangle$.
Let $\theta$ be the projection of  $G_k$ onto $G_{k-1}$. Assume that $\theta':FG_{k}\rightarrow FG_{k-1}$
is an algebra homomorphism over $F$ which is a linear extension of the projection $\theta$.
Obviously, the kernel $\mathrm{Ker}\theta'$ of homomorphism is an ideal of $FG_k$ generated by  $a_k-1$.
Under the map $\theta'$, the image of a unitary normalized unit in $V(FG_k)$ is a unitary normalized unit in
$V(FG_{k-1})$. Further, since $\theta$ is an epimorphism fixing $G_{k-1}$, the restricted map $\theta'|_{V_{*}(F(G_k))}$
is an epimorphism from $V_{*}(FG_k)$ onto $V_{*}(FG_{k-1})$, and its kernel $A^{*}_k=(1+\mathrm{Ker}\theta')\bigcap V_{*}(FG_k)$
is an elementary abelian group. From this, for arbitrary $x\in A^*_k$, we have $x^*=x$. Suppose that $x=1+\sum\limits_{g\in G_{k-1}}\alpha_gg(a_k-1)$,
then $$x=1+\sum\limits_{g\in \Omega_1(G_{k-1})}\alpha_gg(a_k-1)+\sum\limits_{g\in I}\alpha_g(g+g^{-1})(a_k-1),$$
where $I$ is a subset of $G_{k-1}\setminus \Omega_1(G_{k-1})$ such that if $g\in I$ then $g^{-1}\notin I$. Therefore,
the cardinality of $I$ is $2^{k-2}(|G|-|\Omega_1(G)|)$. From this, the order of  $A^*_k$ is $|F|^{2^{k-2}(|G|+|\Omega_1(G)|)}$.

Since there is an inclusion map $i:V_*(FG_{k-1})\rightarrow V_*(FG_k)$, we have $V_*(FG_k)=A^*_k\rtimes V_*(FG_{k-1})$.
Assume that $A^*_i=(1+\mathrm{Ker}\theta')\bigcap V_{*}(FG_i)$, then by induction one can prove the lemma.

If $|V_{*}(FG)|=l|F|^{\frac{1}{2}(|G|+|\Omega_1(G)|-1)}$ for some nature number $l$, then
\begin{center}
$\begin{aligned}
|V_{*}(F(G\times A))|&=l|F|^{\frac{1}{2}(|G|+|\Omega_1(G)|)-1}|F|^{\frac{1}{2}(|G|+|\Omega_1(G)|)(|A|-1)}\\
&=l|F|^{\frac{1}{2}(|G||A|+|\Omega_1(G)||A|)-1}\\
&=l|F|^{\frac{1}{2}(|G\times A|+|\Omega_1(G\times A)|)-1}.
\end{aligned}$
\end{center}
\end{proof}

Note that the central product of $Q_8$ and $Q_8$ is isomorphic to the central product of $D_8$ and $Q_8$. From this,
the structure of an extraspecial $p$-group is as follows.

\begin{lemma}[\cite{Robinson}] \label{EP2}
An extraspecial $p$-group is a central product of $n$ nonabelian subgroups of order $p^3$ and has order $p^{2n+1}$.
If $p=2$, $G$ is a central product of $D_8's$ or a central product of $D_8's$ and a single $Q_8$.
\end{lemma}

For an inner abelian finite $p$-group, we have the following results.

\begin{lemma}[\cite{XuQ}] \label{INA}
Let $G$ be a finite $p$-group, then the following properties are equivalent:

(i) $G$ is an inner abelian group;

(ii) $d(G)=2$ and $|G'|=p$.

(iii) $d(G)=2$ and $\zeta G=Frat (G)$.

\end{lemma}

\begin{lemma}[\cite{XuQ}] \label{INA1}
Let $G$ be an inner abelian finite $p$-group, then $G$ is one of the following types:

(i) $Q_8$, the quaternion group;

(ii) $M_p(n,m)=\langle a,b\ |\ a^{p^n}=b^{p^m}=1,a^b=a^{1+p^{n-1}}\rangle$,  $n\geq 2, m\geq 1$;

(iii) $M_p(n,m,1)=\langle a,b,c\ |\ a^{p^n}=b^{p^m}=c^p=1,[a,b]=c,[c,a]=[c,b]=1\rangle$,  $n\geq  m\geq 1$.

\end{lemma}

\section{ Isomorphism types of  the group}

Let $G$ be a nonabelian  $2$-group given by a
central extension of the form
$$1\longrightarrow N \longrightarrow G \longrightarrow \mathbb{Z}_2\times
\cdots\times \mathbb{Z}_2 \longrightarrow 1$$ and $G'=\langle c\rangle\cong
\mathbb{Z}_2$, $N\cong \mathbb{Z}_{2^n}\times \mathbb{Z}_{2^m}$,   $n, m\geq 1$.
Without loss of generality, we  suppose $n\geq m\geq1$  in the following sections.

 Note that $G'\leq\mathrm{Frat}\ G\leq N\leq \zeta G$.
Hence $G/\zeta G$ is an elementary abelian $2$-group. Further, we have the following lemma.

\begin{lemma} \label{STG}
 (1) For any two elements
$\bar{x}=x\zeta G$ and $\bar{y}=y\zeta G$ of $G/\zeta G$, write
$[x,y]=c^{r}$ ($0\leq r\leq 1$) and  $f(\bar{x},\bar{y})=r$, then
$G/\zeta G$ becomes a nondegenerate symplectic space over a field $F$ with  $2$ elements.

(2) $G=E\Y\zeta G$, where $E$ is  a central product of  some inner abelian groups.
Furthermore, these inner abelian groups have the isomorphism classes:
$Q_8$, $M_2(u,v)$ and $M_2(w,1,1)$, where $u,v,w\leq n+1$.

\end{lemma}

\begin{proof}
(1) Obviously  $f$ is well-defined. For  $x$, $y$, $x_i$, $y_i\in
G$, $i=1,2$, we have $[x_1x_2,y]=[x_1,y][x_2,y]$ and
$[x,y_1y_2]=[x,y_1][x,y_2]$, thus $f$ is bilinear. Since $[x,x]=1$
and $[x,y]=[y,x]^{-1}$, $f(\bar{x},\bar{x})=0$ and
$f(\bar{x},\bar{y})=-f(\bar{y},\bar{x})$, so $G/\zeta G$ is a
symplectic
 space over $F$. If $f(\bar{x},\bar{y})=0$ for all $y\in
G$, then $[x,y]=1$, thus $x\in \zeta G$. It follows that the
symplectic space $G/\zeta G$ is nondegenerate.

(2) From (1), we may assume that the dimension of symplectic
space $G/\zeta G$ is $2k$, and
$\{\bar{x}_1,\ldots,\bar{x}_k,\bar{y}_1,\ldots,\\ \bar{y}_k\}$ is a
basis of $G/\zeta G$,  where $\bar{x}_i=x_i\zeta G$ and
$\bar{y}_i=y_i\zeta G$ for $i=1,\ 2,\ \ldots,\ k$, satisfying:
 $f(\bar{x}_i,\bar{y}_i)=1$, that is, $[x_i,y_i]=c$;
For $i\neq j$,
$f(\bar{x}_i,\bar{x}_j)=f(\bar{y}_i,\bar{y}_j)=f(\bar{x}_i,\bar{y}_j)=0$,
that is, $[x_i,x_j]=[y_i,y_j]=[x_i,y_j]=1$.

Let $G_i:=\langle x_i,y_i\rangle$, then for $i\neq j$, we
have that $[G_i,G_j]=1$. Obviously $d( G_i)=2$  and $|G_i'|=2$, also by Lemma \ref{INA},
we have $G_i$ is inner abelian. Note that $G_i^2\leq N$.  According to Lemma \ref{INA1},
the isomorphism classes of $G_i$ are $Q_8$, $M_2(u,v)$ and $M_2(w,1,1)$, where $u,v,w\leq n+1$.

\end{proof}

Let $r:=d(\zeta G)$. Suppose that
\begin{align}\label{3.1}
\zeta G=\langle z_1\rangle\times\langle z_2\rangle\times\cdots\times\langle z_r\rangle\cong\mathbb{Z}_{2^{n_1}}\times \mathbb{Z}_{2^{n_2}}\times\cdots\times\mathbb{Z}_{2^{n_r}},
 n_1\geq n_2\geq\cdots\geq n_r\geq 1.
\end{align}
Since  $(\zeta G)^2\leq N$, we have $n_i=1$ for $3\leq i\leq r$ and $\langle z_1^2\rangle\times\langle z_2^2\rangle=\mathrm{Frat \zeta G}\leq N$.
If $n_2\geq 2$, then
\begin{align}\label{3.2}
N\bigcap \left(\langle z_3\rangle\times\cdots\times\langle z_r\rangle\right)=1,
\end{align}
otherwise, $|N[2]|\geq 2^3$, a contradiction.
If $n_2=1$,  then  $m$ must be $1$ by $N\leq \zeta G$.  At this time,  by adjusting the parameters $z_1,z_2,\ldots,z_r$, we similarly may obtain
\eqref{3.2}. Further we may obtain the following lemma.

\begin{lemma} \label{SCG}
(1) $n\leq n_1\leq n+1$ and $m\leq n_2\leq m+1$.

(2) The isomorphism classes of  $\zeta G$ are as follows:

 (i) $A_1\cong\mathbb{Z}_{2^{n}}\times \mathbb{Z}_{2^{m}}\times\underbrace{\mathbb{Z}_2\times\cdots\times\mathbb{Z}_2}_{r-2}$.
 In this case, $N=\langle z_1\rangle\times\langle z_2\rangle $.

(ii) $A_2\cong\mathbb{Z}_{2^{n+1}}\times \mathbb{Z}_{2^{m}}\times\underbrace{\mathbb{Z}_2\times\cdots\times\mathbb{Z}_2}_{r-2}$.
 In this case, $N=\langle z_1^2\rangle\times\langle z_2\rangle $.

(iii) $A_3\cong\mathbb{Z}_{2^{n}}\times \mathbb{Z}_{2^{m+1}}\times\underbrace{\mathbb{Z}_2\times\cdots\times\mathbb{Z}_2}_{r-2}$.
 In this case, $N=\langle z_1\rangle\times\langle z_2^2\rangle $.

(iv) $A_4\cong\mathbb{Z}_{2^{n+1}}\times \mathbb{Z}_{2^{m+1}}\times\underbrace{\mathbb{Z}_2\times\cdots\times\mathbb{Z}_2}_{r-2}$.
 In this case, $N=\langle z_1^2\rangle\times\langle z_2^2\rangle $.
\end{lemma}

\begin{proof}
(1) Since $N\leq \zeta G$, $n=\mathrm{Exp}N\leq \mathrm{Exp}(\zeta G)=n_1$. Also since $(\zeta G)^2\leq N$,
$n_1-1\leq n$. From this, we have $n\leq n_1\leq n+1$.

If $n_2\leq m-1$, then $N\leq \zeta G$ implies that
 $\mathbb{Z}_{2^{n-m+1}}\times\mathbb{Z}_{2}\cong N^{2^{m-1}}\leq (\zeta G)^{2^{m-1}}\cong \mathbb{Z}_{2^{n_1-m+1}}$,
 a contradiction. Hence $n_2\geq m$.

Suppose that $n_1=n+1$.
Note that $\mathbb{Z}_{2^{n_1-1}}\times\mathbb{Z}_{2^{n_2-1}}\cong (\zeta G)^2\leq N$.  If $n_2\geq m+2$, then
$|N|\geq 2^{n+m+1}$, a contradiction. From this, we have $m\leq n_2\leq m+1$.
 Suppose that $n_1=n$.
 If $n_2=m+2$, then $(\zeta G)^2\cong\mathbb{Z}_{2^{n-1}}\times\mathbb{Z}_{2^{m+1}}\cong N\cong\mathbb{Z}_{2^n}\times\mathbb{Z}_{2^m}$. From this, we have $n-1=m$ and $n_1=n=m+1<n_2$, a contradiction.
If $n_2\geq m+3$, then $2^{n+m+1}\leq|(\zeta G)^2|\leq |N|=2^{n+m}$, a contradiction.  In a word, (1) is true.

(2) We will next distinguish the isomorphism classes of $\zeta G$. Let $N=\langle a\rangle\times\langle b\rangle$, where $|a|=2^n$ and $|b|=2^m$.

(i) Assume that  $n_1=n$ and $n_2=m$. By \eqref{3.2}, we obtain $\zeta G=\langle a,b,z_3,\ldots,z_r\rangle$, which is the isomorphism class $A_1$ of $\zeta G$.

(ii) Assume that  $n_1=n+1$ and $n_2=m$. We may let $z_1^2=a^ib^j$ since $\langle z_1^2\rangle\leq N$.
If $n>m$, then $i$ and $2$ are coprime  since $|z_1^2|=2^n=|a|>|b|$. In this case, we have
$N=\langle a^ib^j\rangle\times\langle b\rangle= \langle z_1^2\rangle\times\langle b\rangle$ and $\zeta G=\langle z_1,b,z_3,\ldots,z_r\rangle$ by \eqref{3.2}, which is  the structure $A_2$ of  $\zeta G$.

Suppose $n=m$. Obviously,  one of $i$ and $j$ must be coprime to $2$. If $i$ is coprime to $2$, then
we have the structure $A_2$ of  $\zeta G$ which is similar to the case $n>m$.
 If $j$ is coprime to $2$, then
 $N=\langle a^ib^j\rangle\times\langle a\rangle= \langle z_1^2\rangle\times\langle a\rangle$ and $\zeta G=\langle z_1,a,z_3,\ldots,z_r\rangle$ by \eqref{3.2} which is the structure $A_2$ of  $\zeta G$.

(iii) Assume that  $n_1=n$ and $n_2=m+1$. Not to cause confusion, we may similarly  let $z_1^2=a^ib^j$ and $z_2^2=a^ub^v$.
Obviously, $i$ must be divisible by $2$ and let $i=2i_1$.

If $v$ is coprime to $2$, then $N=\langle a\rangle\times\langle z_2^2\rangle $ and $\zeta G=\langle a,z_2,z_3,\ldots,z_r\rangle$ by \eqref{3.2},
which is the isomorphism class $A_3$ of $\zeta G$.

Suppose that $v$ is divisible by $2$ and let $v=2v_1$. Obviously $z_2^{2^m}=a^{2^{m-1}u}$ is of order $2$.
Since $n=n_1\geq n_2=m+1$, we have $n>m$.
If $j$ is divisible by $2$, then
$z_1^{2^{n-1}}=a^{2^{n-1}i_1}$ of order 2 is equal to $a^{2^{m-1}u}$, a contradiction.
It follows that $(j,2)=1$. Further  we have that $n=m+1$, otherwise, $\langle z_1\rangle \bigcap \langle z_2\rangle=\langle a^{2^{m-1}u} \rangle\neq 1$,
a contradiction.
 Hence $(a^ib^j)^{2^{m-1}}$ is an element with order $2$ of $\langle z_1\rangle$.
 From this, we may obtain $N=\langle a\rangle\times\langle a^ib^j\rangle=\langle a\rangle\times\langle z_1^2\rangle$ and $\zeta G=\langle a,z_1,z_3,\ldots,z_r\rangle$, which is the isomorphism class $A_3$ of $\zeta G$ by adjusting the parameters $z_1$ and $z_2$.

(iv) Assume that $n_1=n+1$ and $n_2=m+1$. Since $(\zeta G)^2=\langle z_1^2\rangle\times\langle z_2^2\rangle$ has order $2^{n+m}$,
we have $N=(\zeta G)^2$. In this case, we may take $a=z_1^2$ and $b=z_2^{2}$, which is the isomorphism class $A_4$ of $\zeta G$.

\end{proof}

According to Lemma \ref{STG}, we know that $G$ is the central product of $E$ and $\zeta G$. Further, suppose that $E$ is the central product of $G_1,G_2,\ldots,G_k$, where the isomorphism classes of $G_i$ ($i=1,2,\ldots, k$) are  $Q_8$, $M_2(u,v)$ and $M_2(w,1,1)$, where $u,v,w\leq n+1$  as in Lemma  \ref{STG}.
Next we will determine the isomorphism classes of $G$ according to the types of $\zeta G$ in Lemma \ref{SCG}.

 \subsection{ The isomorphism type $A_1$ of $\zeta G$}

In the section, let
$$\zeta G=\langle z_1\rangle\times\langle z_2\rangle\times\cdots\times\langle z_r\rangle\cong\mathbb{Z}_{2^{n}}\times \mathbb{Z}_{2^{m}}\times\mathbb{Z}_2\times\cdots\times\mathbb{Z}_2, N=\langle z_1\rangle\times\langle z_2\rangle.$$
Obviously, $c\in N[2]=\langle z_1^{2^{n-1}}\rangle\times\langle z_2^{2^{m-1}}\rangle$. If $c=z_1^{2^{n-1}}\cdot z_2^{2^{m-1}}$,
 then we may rewrite $N=\langle z_1\rangle\times\langle z_1^{2^{n-m}}z_2\rangle$. At this time, $c=(z_1^{2^{n-m}}z_2)^{2^{m-1}}$.
Without loss of generality, we may always suppose $c\in \langle z_1\rangle$ or $c\in \langle z_2\rangle$.

For every factor $G_i$ of the central product of $E$,  we  first determine the types of $G_i\Y N$.
Note that $\zeta G_i\leq N$ in the central product $G_i\Y N$. For convenience,  the notations as $M_2(m+1,1,1)\Y \mathbb{Z}_{2^n}$,  $D_8\Y \mathbb{Z}_{2^{n}}$, $M_2(m+1,1,1)\Y M_2(n+1,1) $ and so on, imply the intersections of the factors of  the central products  are  the group $\langle c \rangle$ in the following parts.

\begin{lemma} \label{CPC1}
Suppose $n\geq m\geq 1$ and $c\in \langle z_1\rangle$.

(i) If $G_i\cong M_2(n+1,m+1)$ , then $G_i\Y N=G_i$. When $n>m$, $G_i\ncong M_2(m+1,n+1)$.

(ii) If $G_i\cong M_2(n+1,v)$, then $1\leq v\leq m+1$. When $v\leq m$,  $G_i\Y N\cong M_2(n+1,1)\times \mathbb{Z}_{2^m}$.

(iii) If $G_i\cong M_2(u,n+1)$, then $2\leq u\leq m+1$ and $n=m$.  When $u< m+1$, $G_i\Y N\cong M_2(m+1,1,1)\Y\mathbb{Z}_{2^m}$.

(iv) If $G_i\cong M_2(u,v)$, where $2\leq u<n+1$ and $1\leq v<n+1$, then $G_i\Y N\cong M_2(m+1,1,1)\Y \mathbb{Z}_{2^n}$ or $D_8\Y \mathbb{Z}_{2^{n}}\times \mathbb{Z}_{2^{m}}$.

(v) If $G_i\cong M_2(n+1,1,1)$, then $n=m$ and $G_i\Y N\cong M_2(m+1,1,1)\Y \mathbb{Z}_{2^m}$.

(vi) If $G_i\cong M_2(w,1,1)$, where $1\leq w<n+1$, then $G_i\Y N\cong M_2(m+1,1,1)\Y \mathbb{Z}_{2^n}$ or $D_8\Y \mathbb{Z}_{2^{n}}\times \mathbb{Z}_{2^{m}}$.

\end{lemma}

\begin{proof}
(i) Let
$$G_i=\langle x_i,y_i \ | \ x_i^{2^{n+1}}=y_i^{2^{m+1}}=1, x_i^{y_i}=x_i^{1+2^n} \rangle \cong M_2(n+1,m+1).$$
Since $\zeta G_i=\langle x_i^2\rangle\times\langle y_i^2\rangle\cong \mathbb{Z}_{2^{n}}\times\mathbb{Z}_{2^{m}}$,
we have $N=\zeta G_i$ and  $G_i\Y N=G_i$.

If
$$G_i=\langle x_i,y_i \ | \ x_i^{2^{m+1}}=y_i^{2^{n+1}}=1, x_i^{y_i}=x_i^{1+2^m} \rangle \cong M_2(m+1,n+1),$$
then $G_i'=\langle x_i^{2^m}\rangle$. Thus  $x_i^{2^m}=c=z_1^{2^{n-1}}$.
Since $\zeta G_i=\langle x_i^2\rangle\times\langle y_i^2\rangle\leq N$,
we have $N=\zeta G_i$ by comparing their orders. But, the case when $n>m$ implies that
$$\langle c\rangle=N^{2^{n-1}}=(\zeta G_i)^{2^{n-1}}=\langle y_i^{2^{n}}\rangle,$$  which is
a contradiction. Hence $G_i\ncong M_2(m+1,n+1)$.

(ii) Let
$$G_i=\langle x_i,y_i \ | \ x_i^{2^{n+1}}=y_i^{2^{v}}=1, x_i^{y_i}=x_i^{1+2^n} \rangle \cong M_2(n+1,v),$$
then $G_i'=\langle x_i^{2^n}\rangle$. Thus  $x_i^{2^n}=c=z_1^{2^{n-1}}$.
Since $\zeta G_i=\langle x_i^2\rangle\times\langle y_i^2\rangle\leq N$, we have $2^{n+v-1}\leq |N|=2^{n+m}$. It follows
$v-1\leq m$.

Suppose $v\leq m$. Let $ x_i^2=z_1^{l}z_2^{j}$ and $ y_i^2=z_1^{s}z_2^{t}$, where
$1\leq l,s\leq 2^n, 1\leq j,t\leq 2^m$. If $l$ is divisible by $2$, then $c=x_i^{2^n}=z_1^{2^{n-1}l}z_2^{2^{n-1}j}=z_2^{2^{n-1}j}$, which is a contradiction. From this, $l$ must be coprime to $2$.  Since $1=y_i^{2^v}=z_1^{2^{v-1}s}z_2^{2^{v-1}t}$, we have
$s$ and $t$ are divisible by $2^{n-v+1}$ and $2^{m-v+1}$, respectively. Let $s=2s_1$ and $t=2t_1$.
 Hence $y_iz_1^{-s_1}z_2^{-t_1}$ is of order $2$. At this time,  we may replace $y_i$ by $y_iz_1^{-s_1}z_2^{-t_1}$.
$$ G_i\Y N=\langle x_i,y_iz_1^{-s_1}z_2^{-t_1},z_2 \rangle=\langle x_i,y_iz_1^{-s_1}z_2^{-t_1}\rangle\times\langle z_2 \rangle\cong M_2(n+1,1)\times\mathbb{Z}_{2^m}.$$

(iii) Let
$$G_i=\langle x_i,y_i \ | \ x_i^{2^{u}}=y_i^{2^{n+1}}=1, x_i^{y_i}=x_i^{1+2^{u-1}} \rangle \cong M_2(u,n+1).$$
Since $\zeta G_i=\langle x_i^2\rangle\times\langle y_i^2\rangle\leq N$, we have $2^{u-1+n}\leq |N|=2^{n+m}$. It follows
$u-1\leq m$.  Let $ x_i^2=z_1^{l}z_2^{j}$ and $ y_i^2=z_1^{s}z_2^{t}$, where
$1\leq l,s\leq 2^n, 1\leq j,t\leq 2^m$.

If $n>m$,  then $1\neq y_i^{2^n}=z_1^{2^{n-1}s}z_2^{2^{n-1}t}=z_1^{2^{n-1}s}=c=x_i^{2^{u-1}}$, which is a contradiction.
Hence we obtain $n=m$.
Suppose that $u< m+1$.  Since $1=x_i^{2^u}=z_1^{2^{u-1}l}z_2^{2^{u-1}j}$, we have both
$l$ and $j$ are divisible by $2^{m-u+1}$. Let $l=2l_1$ and $j=2j_1$.
If $t$ is divisible by $2$, then $1\neq y_i^{2^n}=z_1^{2^{n-1}s}z_2^{2^{n-1}t}=z_1^{2^{n-1}s}=c$, which is a contradiction.
Thus we have $t$ is coprime to $2$. In a word,  we may replace $x_i$ by $x_iz_1^{-l_1}z_2^{-j_1}$ and
$$ G_i\Y N=\langle x_iz_1^{-l_1}z_2^{-j_1},y_i,z_1 \rangle=\langle x_iz_1^{-l_1}z_2^{-j_1},y_i\rangle\Y\langle z_1 \rangle\cong M_2(m+1,1,1)\Y\mathbb{Z}_{2^m}.$$

(iv)
Let
$$G_i=\langle x_i,y_i \ | \ x_i^{2^{u}}=y_i^{2^{v}}=1, x_i^{y_i}=x_i^{1+2^{u-1}} \rangle \cong M_2(u,v),$$
where $u,v\leq n$.
 Let $ x_i^2=z_1^{l}z_2^{j}$ and $ y_i^2=z_1^{s}z_2^{t}$, where
$1\leq l,s\leq 2^n, 1\leq j,t\leq 2^m$.
Obviously, both $l$ and $s$ are divisible by $2$. Let $l=2l_1$ and $s=2s_1$.

If $(j,2)=1=(t,2)$, then there exist $j_1$ and $t_1$ such that $jj_1\equiv 1\pmod{2^m}$ and $tt_1\equiv 1\pmod{2^m}$.
Hence $x_i^{2j_1}=z_1^{2l_1j_1}z_2$ and $y_i^{2t_1}=z_1^{2s_1t_1}z_2$.  Note that $n\geq u\geq 2$ and $c=z_1^{2^{n-1}}$. From this, we have
$(x_i^{j_1}y_i^{-t_1}z_1^{s_1t_1-l_1j_1+2^{n-2}})^2=1$. It follows that
$$G_i\Y N=\langle y_iz_1^{-s_1}, x_i^{j_1}y_i^{-t_1}z_1^{s_1t_1-l_1j_1+2^{n-2}}, z_1\rangle\cong M_2(m+1,1,1)\Y \mathbb{Z}_{2^n}.$$

Suppose that  $(j,2)=1$ and $2|t$. Let $t=2t_2$. Since the orders of $x_iz_1^{-l_1}$ and $y_iz_1^{-s_1}z_2^{-t_2}$ are $2^{m+1}$ and $2$,  respectively, we have
$$G_i\Y N=\langle x_iz_1^{-l_1}, y_iz_1^{-s_1}z_2^{-t_2},z_1\rangle\cong M_2(m+1,1,1)\Y \mathbb{Z}_{2^n}.$$
Similarly, we may obtain the same result for the case when  $2|j$ and $(t,2)=1$.

Suppose that  $2|j$ and $2|t$. Let $j=2j_2$ and $t=2t_3$. We may obtain
$$G_i\Y N=\langle x_iz_1^{-l_1}z_2^{-j_2}, y_iz_1^{-s_1}z_2^{-t_3},z_1,z_2\rangle\cong D_8\Y \mathbb{Z}_{2^n}\times \mathbb{Z}_{2^m}. $$

(v)
Let
$$G_i=\langle x_i,y_i ,c\ | \ x_i^{2^{n+1}}=y_i^2=c^2=1, [x_i,y_i]=c,[x_i,c]=[y_i,c]=1 \rangle \cong M_2(n+1,1,1).$$
 Let $ x_i^2=z_1^{l}z_2^{j}$, where
$1\leq l\leq 2^n, 1\leq j\leq 2^m$.
If $n>m$, then $x_i^{2^n}=z_1^{2^{n-1}l}z_2^{2^{n-1}j}=z_1^{2^{n-1}l}$, which is a contradiction.
Hence $n=m$. At this time,  if both $l$ and $j$ can be divisible by $2$, then the order of $z_1^{l}z_2^{j}$ is less than $2^m$, which
is impossible. If $j$ can be divisible by $2$ and $(l,2)=1$, then we have $x_i^{2^m}=z_1^{2^{m-1}l}z_2^{2^{m-1}j}=c$,
which is also impossible. From this, we have $(j,2)=1$. It follows that
$$G_i\Y N=\langle x_i, y_i,z_1\rangle\cong M_2(m+1,1,1)\Y \mathbb{Z}_{2^m}.$$

(vi)
Let
$$G_i=\langle x_i,y_i ,c\ | \ x_i^{2^{w}}=y_i^2=c^2=1, [x_i,y_i]=c,[x_i,c]=[y_i,c]=1 \rangle \cong M_2(w,1,1),$$
 where $1\leq w<n+1$. Let $ x_i^2=z_1^{l}z_2^{j}$, where
$1\leq l\leq 2^n, 1\leq j\leq 2^m$.
Since $1=x_i^{2^w}=z_1^{2^{w-1}l}z_2^{2^{w-1}j}$, we have $l$ can be divisible by $2$. Let $l=2l_1$.

If $(j,2)=1$, then $G_i\Y N=\langle x_iz_1^{-l_1}, y_i,z_1\rangle\cong M_2(m+1,1,1)\Y \mathbb{Z}_{2^n}.$

Suppose $j$ can be divisible by $2$ and let $j=2j_1$. Thus
$$G_i\Y N=\langle x_iz^{-l_1}z_2^{-j_1}, y_i,z_1,z_2\rangle\cong D_8\Y \mathbb{Z}_{2^n}\times\mathbb{Z}_{2^m}.$$

\end{proof}

\begin{lemma} \label{CPC2}
Suppose $n\geq m\geq 1$ and $c\in \langle z_2\rangle$.

(i) If $G_i\cong M_2(m+1,n+1)$ , then $G_i\Y N=G_i$. When $n>m$, $G_i\ncong M_2(n+1,m+1)$.

(ii) If $G_i\cong M_2(n+1,v)$, where $1\leq v\leq m+1$, then $n=m$. When $1\leq v\leq m$,  $G_i\Y N\cong M_2(m+1,1)\times \mathbb{Z}_{2^m}$.

(iii) If $G_i\cong M_2(u,n+1)$, where $2\leq u\leq m$, then $G_i\Y N\cong M_2(n+1,1,1)\Y\mathbb{Z}_{2^m}$.

(iv) If $G_i\cong M_2(u,v)$, where $2\leq u<n+1$ and $1\leq v<n+1$, then
$$G_i\Y N\cong\left\{ \begin{aligned}
          &D_8\times \mathbb{Z}_{2^n}\  or\  Q_8\times \mathbb{Z}_{2^{n}},\ \ m=1.\\
          &M_2(m+1,1)\times \mathbb{Z}_{2^n}\ \mathrm{or}\  D_8\Y \mathbb{Z}_{2^{m}}\times \mathbb{Z}_{2^{n}},\ \ m>1.
                                           \end{aligned} \right.$$

(v) If $G_i\cong M_2(n+1,1,1)$, then
$G_i\Y N\cong  M_2(n+1,1,1)\Y \mathbb{Z}_{2^m}.$

(vi) If $G_i\cong M_2(w,1,1)$, where $1\leq w<n+1$, then $G_i\Y N\cong M_2(m+1,1)\times \mathbb{Z}_{2^n}$ or $D_8\Y \mathbb{Z}_{2^{m}}\times \mathbb{Z}_{2^{n}}$.

\end{lemma}

\begin{proof}
(i) - (iii) may be obtained similar to  (i) -(iii) of Lemma \ref{CPC1}.

(iv)
Let
$$G_i=\langle x_i,y_i \ | \ x_i^{2^{u}}=y_i^{2^{v}}=1, x_i^{y_i}=x_i^{1+2^{u-1}} \rangle \cong M_2(u,v),$$
where $u,v\leq n$.
 Let $ x_i^2=z_1^{l}z_2^{j}$ and $ y_i^2=z_1^{s}z_2^{t}$, where
$1\leq l,s\leq 2^n, 1\leq j,t\leq 2^m$.
Obviously, both $l$ and $s$ are divisible by $2$. Let $l=2l_1$ and $s=2s_1$.

If $(j,2)=1=(t,2)$, then there exist $j_1$ and $t_1$ such that $jj_1\equiv 1\pmod{2^m}$ and $tt_1\equiv 1\pmod{2^m}$.
Hence $x_i^{2j_1}=z_1^{2l_1j_1}z_2$ and $y_i^{2t_1}=z_1^{2s_1t_1}z_2$.
 Further,
$(x_i^{j_1}y_i^{-t_1}z_1^{s_1t_1-l_1j_1}z_2^{2^{m-2}})^2=1$ when $m\geq 2$. It follows that
$$G_i\Y N=\langle x_iz_1^{-l_1}, x_i^{j_1}y_i^{-t_1}z_1^{s_1t_1-l_1j_1}z_2^{2^{m-2}}, z_1\rangle\cong M_2(m+1,1)\times \mathbb{Z}_{2^n}.$$
When $m=1$, we have $$G_i\Y N=\langle x_iz_1^{-l_1}, y_iz_1^{-s_1}, z_1\rangle\cong Q_8\times \mathbb{Z}_{2^n}.$$

Suppose that  $(j,2)=1$ and $2|t$. Let $t=2t_2$. Since the orders of $x_iz_1^{-l_1}$ and $y_iz_1^{-s_1}z_2^{-t_2}$ are $2^{m+1}$ and $2$, we have
$$G_i\Y N=\langle x_iz_1^{-l_1}, y_iz_1^{-s_1}z_2^{-t_2},z_1\rangle\cong M_2(m+1,1)\times \mathbb{Z}_{2^n}.$$
Similarly, we may obtain the same result for the case $2|j$ and $(t,2)=1$.

Suppose that  $2|j$ and $2|t$. Let $j=2j_2$ and $t=2t_3$. We may obtain
$$G_i\Y N=\langle x_iz_1^{-l_1}z_2^{-j_2}, y_iz_1^{-s_1}z_2^{-t_3},z_2,z_1\rangle\cong D_8\Y \mathbb{Z}_{2^m}\times \mathbb{Z}_{2^n}. $$

(v)
Let
$$G_i=\langle x_i,y_i ,c\ | \ x_i^{2^{n+1}}=y_i^2=c^2=1, [x_i,y_i]=c,[x_i,c]=[y_i,c]=1 \rangle \cong M_2(n+1,1,1).$$
 Let $ x_i^2=z_1^{l}z_2^{j}$, where
$1\leq l\leq 2^n, 1\leq j\leq 2^m$.
If both $l$ and $j$ can be divisible by $2$, then the order of $z_1^{l}z_2^{j}$ is less than $2^{n}$, which is impossible.
If $l$ can be divisible by $2$, then $(j,2)=1$ and $x_i^{2^n}=z_1^{2^{n-1}l}z_2^{2^{n-1}j}=z_2^{2^{n-1}j}$, which is a contradiction.
Hence  $(l,2)=1$. From this, we have
$$G_i\Y N=\langle x_i, y_i,z_2\rangle\cong M_2(n+1,1,1)\Y \mathbb{Z}_{2^m}. $$

(vi)
Let
$$G_i=\langle x_i,y_i ,c\ | \ x_i^{2^{w}}=y_i^2=c^2=1, [x_i,y_i]=c,[x_i,c]=[y_i,c]=1 \rangle \cong M_2(w,1,1),$$
 where $1\leq w<n+1$. Let $ x_i^2=z_1^{l}z_2^{j}$, where
$1\leq l\leq 2^n, 1\leq j\leq 2^m$.
Since $1=x_i^{2^w}=z_1^{2^{w-1}l}z_2^{2^{w-1}j}$, we have $l$ can be divisible by $2$. Let $l=2l_1$.

If $(j,2)=1$, then $G_i\Y N=\langle x_iz_1^{-l_1}, y_i,z_1\rangle\cong M_2(m+1,1)\times \mathbb{Z}_{2^n}.$

Suppose $j$ can be divisible by $2$ and let $j=2j_1$. Thus
$$G_i\Y N=\langle x_iz^{-l_1}z_2^{-j_1}, y_i, z_2, z_1\rangle\cong D_8\Y \mathbb{Z}_{2^m}\times\mathbb{Z}_{2^n}.$$

\end{proof}

According to Lemmas \ref{STG} and \ref{CPC1}, for arbitrary  factor $G_i$ of the central product of $E$, $G_i\Y N$ has the
following isomorphism classes: $M_2(n+1,m+1)$, $M_2(n+1,1)\times \mathbb{Z}_{2^m}$, $M_2(m+1,1,1)\Y \mathbb{Z}_{2^n}$,
$D_8\Y \mathbb{Z}_{2^n}\times\mathbb{Z}_{2^m}$ and $Q_8\Y \mathbb{Z}_{2^n}\times\mathbb{Z}_{2^m}$. From this,
 the central product of arbitrary two factors $G_1$ and $G_2$ may only be considered that among each other of
$M_2(n+1,m+1)$, $M_2(n+1,1)$, $M_2(m+1,1,1)$, $D_8$ and $Q_8$. Note that $Q_8\Y Q_8\cong D_8\Y D_8$.
Hence we only consider other cases.

\begin{lemma} \label{GGC1}
Suppose $n\geq m\geq 1$ and $c\in \langle z_1\rangle$. Let $G_1$ and $G_2$ are arbitrary two factors which are isomorphic to  $M_2(n+1,m+1), M_2(n+1,1)$  or $ M_2(m+1,1,1)$.

(i) If $G_1$ and $G_2$ are isomorphic to $M_2(n+1,m+1)$, then $$G_1\Y G_2\cong\left\{ \begin{aligned}
          &M_2(2,1,1)\Y D_8\  \mbox{or}\  M_2(2,1,1)\Y Q_8,\ \ n=1.\\
          & M_2(m+1,1,1)\Y M_2(n+1,1), \ \ n>1.
                                           \end{aligned} \right.$$
(ii) If $G_1$ and $G_2$ are isomorphic to $M_2(n+1,m+1)$ and $M_2(n+1,1)$, respectively, then $G_1\Y G_2\cong M_2(n+1,m+1)\Y D_8$.

(iii) If $G_1$ and $G_2$ are isomorphic to $M_2(n+1,m+1)$ and $M_2(m+1,1,1)$, respectively, then $G_1\Y G_2\cong M_2(n+1,m+1)\Y D_8$.

(iv)  If both $G_1$ and $G_2$ are isomorphic to $M_2(m+1,1,1)$, then $G_1\Y G_2\cong M_2(m+1,1,1)\Y D_8$ or $M_2(m+1,1,1)\Y M_2(m+1,1)$.

(v) If both $G_1$ and $G_2$ are isomorphic to $M_2(n+1,1)$, then $G_1\Y G_2\cong  M_2(n+1,1)\Y D_8$ or $M_2(n+1,1)\Y M_2(m+1,1,1)$.

\end{lemma}

\begin{proof}
(i) Let
$$G_i=\langle x_i,y_i\ | \ x_i^{2^{n+1}}=y_i^{2^{m+1}}=1, x_i^{y_i}=x_i^{1+2^n} \rangle\cong M_2(n+1,m+1),$$
where  $i=1,2$.
Since $\zeta G_1=\zeta G_2=N$, we may let $x_2^2=x_1^{2l}y_1^{2j}$ and
$y_2^2=x_1^{2s}y_1^{2t}$.

First, we suppose that $n>m$. Obviously, $(l,2)=1$ and $s$ is divisible by $2$. Let $s=2s_1$.
If $t$ is divisible by $2$, then $1\neq y_2^{2^m}=x_1^{2^ms}y_1^{2^mt}=x_1^{2^ms}$, which is impossible.
Hence we obtain that $(t,2)=1$.
If $(j,2)=1$, then $x_1^ly_1^jx_2^{-2^{n-1}-1}$ is of order $2$. From this, we have
$$G_1\Y G_2=\langle y_1, x_1^ly_1^jx_2^{-2^{n-1}-1} \rangle\Y\langle x_2,y_2^{-1}x_1^{2s_1}y_1^{t} \rangle\cong M_2(m+1,1,1)\Y M_2(n+1,1) .$$
If $j$ is divisible by $2$, then  we have
$$G_1\Y G_2=\langle y_1, x_1^ly_1^jx_2^{-1} \rangle\Y\langle x_2,y_2^{-1}x_1^{2s_1}y_1^{t} \rangle\cong M_2(m+1,1,1)\Y M_2(n+1,1) .$$

We next suppose that $n=m$.  Obviously, $l$ and $j$ can not simultaneously be divisible by $2$,  nor can $s$ and $t$.
If $j$ and $2$ are  coprime, then $c= x_2^{2^m}=x_1^{2^ml}y_1^{2^mj}=c^ly_1^{2^m}$, which
 is a contradiction. From this, we have $j$ is divisible by $2$ and $(l,2)=1$. Let $j=2j_1$.
Similarly,  $t$ are coprime to $2$.

Suppose that  $s$ is coprime to $2$. If $m\geq 2$, then
 $$G_1\Y G_2=\langle y_2,x_1^{l}y_1^{2j_1}x_2^{-1}\rangle\Y\langle x_1 , x_1^{s-2^{m-1}}y_1^{t}y_2^{-1}\rangle\cong M_2(m+1,1,1)\Y M_2(m+1,1).$$
If $m=1$, then
 $$G_1\Y G_2=\langle y_2,x_1^{l}y_1^{2j_1}x_2^{-1}\rangle\Y\langle x_1 , y_1^{t}y_2^{-1}\rangle\cong M_2(2,1,1)\Y Q_8.$$

If $s$ is divisible by $2$,  then we similarly have $G_1\Y G_2\cong M_2(m+1,1,1)\Y M_2(m+1,1)$.

(ii) Let
$$G_1=\langle x_1,y_1 \ | \ x_1^{2^{n+1}}=y_1^{2^{m+1}}=1, x_1^{y_1}=x_1^{1+2^n} \rangle\cong M_2(n+1,m+1)$$
 and
$$G_2=\langle x_2,y_2\ | \ x_2^{2^{n+1}}=y_2^{2}=1, x_2^{y_2}=x_2^{1+2^n} \rangle\cong M_2(n+1,1).$$
Since $\zeta G_1=N\geq \zeta G_2=\langle x_2^2\rangle$, we may let $x_2^2=x_1^{2l}y_1^{2j}$.
 Further,  we have $(l,2)=1$ since  $c=x_2^{2^n}=x_1^{2^nl}y_1^{2^nj}=c^ly_1^{2^nj}$.

Suppose that $n>m$.  If $(j,2)=1$, then we have
$$G_1\Y G_2=\langle x_1y_2,y_1y_2 \rangle\Y\langle x_2^{-2^{n-1}-1}x_1^{l}y_1^{j},y_2 \rangle\cong M_2(n+1,m+1)\Y D_8.$$
If $j$ is divisible by $2$, then  we have
$$G_1\Y G_2=\langle x_1, y_1y_2 \rangle\Y\langle x_2^{-1}x_1^{l}y_1^{j},y_2 \rangle\cong M_2(n+1,m+1)\Y D_8.$$

If $n=m$, then $j$ is divisible by $2$. Similarly, the result is true.

(iii) Let
$$G_1=\langle x_1,y_1 \ | \ x_1^{2^{n+1}}=y_1^{2^{m+1}}=1, x_1^{y_1}=x_1^{1+2^n} \rangle\cong M_2(n+1,m+1)$$
 and
$$G_2=\langle x_2,y_2,c\ | \ x_2^{2^{m+1}}=y_2^{2}=c^2=1,[x_2,y_2]=c,[x_2,c]=[y_2,c]=1 \rangle\cong M_2(m+1,1,1).$$
Since $\zeta G_1=N\geq \zeta G_2=\langle x_2^2,c\rangle$, we may let $x_2^2=x_1^{2l}y_1^{2j}$.  If $j$ can be divisible by $2$, then $x_2^{2^m}=x_1^{2^ml}y_1^{2^mj}=x_1^{2^ml}=c$, which
 is impossible. From this, we have $(j,2)=1$.

First we suppose that $n>m\geq1$. Obviously $l$ is divisible by $2$. Let $l=2l_1$.
 It follows that
 $$G_1\Y G_2=\langle x_1y_2,y_1 \rangle\Y\langle x_2^{-1}x_1^{2l_1}y_1^{j},y_2 \rangle\cong M_2(n+1,m+1)\Y D_8.$$

We next suppose that $n=m$. For the case when $2|l$, the result is similarly obtained.
If $(l,2)=1$, then we have
$$G_1\Y G_2=\langle x_1y_2,y_1y_2 \rangle\Y\langle x_1^ly_1^{j}x_2^{-1},y_2 \rangle\cong M_2(m+1,m+1)\Y D_8.$$

(iv) Let
$$G_i=\langle x_i,y_i,c\ | \ x_i^{2^{m+1}}=y_i^{2}=c^2=1,[x_i,y_i]=c,[x_i,c]=[y_i,c]=1 \rangle\cong M_2(m+1,1,1),$$
where  $i=1,2$. Let $x_1^2=z_1^{l}z_2^{j}$. If $j$ is divisible by $2$, then $x_1^{2^m}=z_1^{2^{m-1}l}z_2^{2^{m-1}j}=z_1^{2^{m-1}l}$,
which is impossible. Hence we have $(j,2)=1$. Note that $\langle x_1^2\rangle\times\langle z_1\rangle=N=\langle x_2^2\rangle\times\langle z_1\rangle$.
Let $x_2^2=x_1^{2s}z_1^t$. Obviously, $(s,2)=1$.

If $t$ is divisible by $2$, then we may let $t=2t_1$ and  replace $x_1$ by $x_1^{s}z_1^{t_1}$. At this time,
 $ x_1^2=x_2^2$. It follows that
$$G_1\Y G_2=\langle x_1, y_1y_2 \rangle\Y\langle x_2x_1^{-1},y_2 \rangle\cong M_2(m+1,1,1)\Y D_8.$$
If (t,2)=1, then $n=m$ and
$$G_1\Y G_2=\langle x_1, y_1y_2 \rangle\Y\langle x_2x_1^{-s},y_2 \rangle\cong M_2(m+1,1,1)\Y M_2(m+1,1).$$

(v) Let
$$G_i=\langle x_i, y_i\ | \ x_i^{2^{n+1}}=y_i^{2}=1,x_i^{y_i}=x_i^{1+2^n} \rangle\cong M_2(n+1,1),$$
where $i=1,2$. Note that $\langle x_1^2\rangle\times\langle z_2\rangle=N=\langle x_2^2\rangle\times\langle z_2\rangle$.
Let $x_2^2=x_1^{2s}z_2^t$. Obviously, $(s,2)=1$.

If $t$ is divisible by $2$, then we may let $t=2t_1$ and  replace $x_1$ by $x_1^{s}z_2^{t_1}$. At this time,
 $ x_1^2=x_2^2$. It follows that
$$G_1\Y G_2=\langle x_1, y_1y_2 \rangle\Y\langle x_2x_1^{-1},y_2 \rangle\cong M_2(n+1,1)\Y D_8.$$
If (t,2)=1, then
$$G_1\Y G_2=\langle x_1, y_1y_2 \rangle\Y\langle x_2x_1^{-s},y_2 \rangle\cong M_2(n+1,1)\Y M_2(m+1,1,1).$$

\end{proof}

\begin{theorem} \label{ST1}
Suppose that $n\geq m\geq 1$ and $c\in \langle z_1\rangle$. Then the isomorphism classes of $G$ are as follows.

(i) $M_2(n+1,m+1)\Y \underbrace{D_8\Y\cdots\Y D_8}_{k-1}\times\underbrace{\mathbb{Z}_2\times\cdots\times\mathbb{Z}_2}_{r-2}$, $k\geq 1$;

(ii) $M_2(n+1,m+1)\Y Q_8\Y\underbrace{D_8\Y\cdots\Y D_8}_{k-2} \times\underbrace{\mathbb{Z}_2\times\cdots\times\mathbb{Z}_2}_{r-2}$, $k\geq 2$;

(iii) $M_2(n+1,1)\Y M_2(m+1,1,1) \Y\underbrace{D_8\Y\cdots\Y D_8}_{k-2} \times\underbrace{\mathbb{Z}_2\times\cdots\times\mathbb{Z}_2}_{r-2}$, $k\geq 2$;

(iv) $M_2(n+1,1)\Y M_2(m+1,1,1)\Y Q_8 \Y\underbrace{D_8\Y\cdots\Y D_8}_{k-3} \times\underbrace{\mathbb{Z}_2\times\cdots\times\mathbb{Z}_2}_{r-2}$, $k\geq 3$;

(v) $M_2(n+1,1)\Y \underbrace{D_8\Y\cdots\Y D_8}_{k-1}\times \mathbb{Z}_{2^{m}}\times\underbrace{\mathbb{Z}_2\times\cdots\times\mathbb{Z}_2}_{r-2}$, $k\geq 1$;

(vi) $M_2(n+1,1)\Y Q_8 \Y\underbrace{D_8\Y\cdots\Y D_8}_{k-2}\times \mathbb{Z}_{2^{m}}\times\underbrace{\mathbb{Z}_2\times\cdots\times\mathbb{Z}_2}_{r-2}$, $k\geq 2$;

(vii) $M_2(m+1,1,1)\Y\underbrace{D_8\Y\cdots\Y D_8}_{k-1}\Y \mathbb{Z}_{2^{n}}\times\underbrace{\mathbb{Z}_2\times\cdots\times\mathbb{Z}_2}_{r-2}$, $k\geq 1$;

(viii) $M_2(m+1,1,1)\Y Q_8 \Y\underbrace{D_8\Y\cdots\Y D_8}_{k-2}\Y \mathbb{Z}_{2^{n}}\times\underbrace{\mathbb{Z}_2\times\cdots\times\mathbb{Z}_2}_{r-2}$, $k\geq 2$;

(ix) $\underbrace{D_8\Y\cdots\Y D_8}_{k}\Y \mathbb{Z}_{2^{n}}\times \mathbb{Z}_{2^{m}}
\times\underbrace{\mathbb{Z}_2\times\cdots\times\mathbb{Z}_2}_{r-2}$, $k\geq 1$;

(x) $Q_8\Y\underbrace{D_8\Y\cdots\Y D_8}_{k-1}\Y \mathbb{Z}_{2^{n}}\times \mathbb{Z}_{2^{m}}
\times\underbrace{\mathbb{Z}_2\times\cdots\times\mathbb{Z}_2}_{r-2}$, $k\geq 1$.

\end{theorem}

\begin{proof}
First, we rewrite $G$ by $(G_1\Y N)\Y (G_2\Y N)\Y\cdots\Y (G_k\Y N)\Y A_1$. According to Lemmas \ref{STG} and \ref{CPC1},
we have the isomorphism classes of $ G_i\Y N's$ are $M_2(n+1,m+1)$, $M_2(n+1,1)\times \mathbb{Z}_{2^m}$, $M_2(m+1,1,1)\Y \mathbb{Z}_{2^n}$, $D_8\Y \mathbb{Z}_{2^n} \times \mathbb{Z}_{2^m}$ and $Q_8\Y \mathbb{Z}_{2^n} \times \mathbb{Z}_{2^m}$.

First, we suppose that $n>m$.
If $\mathrm{Exp}(E)$ is equal to $2^{n+1}$, then there exists some $G_i$ such that $\mathrm{Exp}(G_i)=2^{n+1}$. Without loss of generality,
suppose $\mathrm{Exp}(G_1)=2^{n+1}$. Thus we have $G_1\Y N\cong M_2(n+1,m+1)$ or $M_2(n+1,1)\times \mathbb{Z}_{2^{m}}$ by Lemma \ref{CPC1}.
If  the number $\lambda$ of $(G_i\Y N)'s$ is odd, which  are isomorphic to $ M_2(n+1,m+1)$, then
we may obtain the isomorphism classes (i) and (ii) of $G$ by Lemma \ref{GGC1}. If $\lambda$ is a nonzero even number, then isomorphism classes of $G$ are (iii) and (iv).  Suppose $\lambda =0$. If  there exists some $G_i\Y N$ such that $G_i\Y N\cong M_2(m+1,1,1)\Y \mathbb{Z}_{2^n}$, then  $M_2(n+1,1)\Y M_2(m+1,1,1)$ will produce $M_2(n+1,m+1)$, which implies $\lambda \neq 0$. At the time, we may obtain the isomorphism classes (v) and (vi) of $G$.

Suppose that $\mathrm{Exp}(E)$ is less than $2^{n+1}$. At this case, the isomorphism classes of $(G_i\Y N)'s$ are $M_2(m+1,1,1)\Y \mathbb{Z}_{2^n}$, $D_8\Y \mathbb{Z}_{2^n} \times \mathbb{Z}_{2^m}$ and $Q_8\Y \mathbb{Z}_{2^n} \times \mathbb{Z}_{2^m}$.
If there exists some $G_i\Y N$ such that $G_i\Y N\cong M_2(m+1,1,1)\Y \mathbb{Z}_{2^n}$, then we may obtain isomorphism classes (vii)
and (viii) of $G$ according to (iv) of Lemma \ref{GGC1} avoiding repeating the above types.  Otherwise, (ix) and (x) may be obtained.

We next suppose that $n=m$. For convenience, we will distinguish the case $m>1$ from $m=1$.
 First, suppose that $m>1$. If $\mathrm{Exp}(E)$ is equal to $2^{m+1}$, then we may
suppose $\mathrm{Exp}(G_1)=2^{m+1}$. Thus we have $G_1\Y N\cong M_2(m+1,m+1)$, $M_2(m+1,1,1)\Y \mathbb{Z}_{2^{m}}$ or  $M_2(m+1,1)\times \mathbb{Z}_{2^m}$ by Lemma \ref{CPC1}.
If  the number $\mu$ of $(G_i\Y N)'s$ is odd, which  are isomorphic to $ M_2(m+1,m+1)$, then
we may obtain the isomorphism classes (i) and (ii) of $G$ by Lemma \ref{GGC1}. If $\mu$ is a nonzero even number, then isomorphism classes of $G$ are (iii) and (iv). Suppose that $\mu=0$. Obviously, $M_2(m+1,1,1)$ and $M_2(m+1,1)$ can not appear simultaneously.
When $M_2(m+1,1)$  comes into being, we may obtain the isomorphism classes (v) and (vi) of $G$. When $M_2(m+1,1,1)$  appears,
(vii) and (viii) may be obtained. Suppose that $\mathrm{Exp}(E)$ is less than $2^{m+1}$. At this case, (ix) and (x) may be obtained.

If $m=1$, then we have $G_i\Y N\cong M_2(2,2)$, $M_2(2,1,1)$, $D_8\times \mathbb{Z}_{2}$ or $Q_8\times \mathbb{Z}_{2}$ by Lemma \ref{CPC1} for arbitrary $i$. At this time, the following arbitrary two are the same classes:  (iii) and (vii), (iv) and (viii),  (v) and (ix), (vi) and (x).
Similarly, we may obtain six isomorphism classes of $G$, that is,  (i), (ii), (iii), (iv),(v) and (vi).

\end{proof}

According to Lemma \ref{CPC2}, we only consider the central products among $M_2(m+1,n+1), M_2(m+1,1), M_2(n+1,1,1), D_8$ and $ Q_8$.
Hence we have the following  lemma.

\begin{lemma} \label{GGC2}
Suppose $n> m\geq 1$ and $c\in \langle z_2\rangle$. Let $G_1$ and $G_2$ are arbitrary two factors which are isomorphic to $M_2(m+1,n+1), M_2(m+1,1)$ or $ M_2(n+1,1,1)$.

(i) If $G_1$ and $G_2$ are isomorphic to $M_2(m+1,n+1)$, then
$$G_1\Y G_2\cong\left\{ \begin{aligned}
          &M_2(n+1,1,1)\Y D_8 \ or\  M_2(n+1,1,1)\Y Q_8,\ \ m=1.\\
          & M_2(n+1,1,1)\Y M_2(m+1,1) \,\ \ m>1.
                                           \end{aligned} \right.$$

(ii) If $G_1$ and $G_2$ are isomorphic to $M_2(m+1,n+1)$ and $M_2(m+1,1)$, respectively, then $G_1\Y G_2\cong M_2(m+1,n+1)\Y D_8$.

(iii) If $G_1$ and $G_2$ are isomorphic to $M_2(m+1,n+1)$ and $M_2(n+1,1,1)$, respectively, then $G_1\Y G_2\cong M_2(m+1,n+1)\Y D_8$.

(iv)  If both $G_1$ and $G_2$ are isomorphic to $M_2(n+1,1,1)$, then $G_1\Y G_2\cong M_2(n+1,1,1)\Y D_8$ or $M_2(n+1,1,1)\Y M_2(m+1,1)$.

(v) If both $G_1$ and $G_2$ are isomorphic to $M_2(m+1,1)$, then $G_1\Y G_2\cong  M_2(m+1,1)\Y D_8$.

\end{lemma}

\begin{proof}
(i) Let
$$G_i=\langle x_i,y_i \ | \ x_i^{2^{m+1}}=y_i^{2^{n+1}}=1, x_i^{y_i}=x_i^{1+2^m} \rangle\cong M_2(m+1,n+1),$$
where  $i=1,2$.
Since $\zeta G_1=\zeta G_2=N$, we may let $x_2^2=x_1^{2l}y_1^{2j}$ and
$y_2^2=x_1^{2s}y_1^{2t}$. Obviously, $j$ is divisible by $2$ and let $j=2j_1$. If $l$ is divisible by $2$, then $c= x_2^{2^m}=x_1^{2^ml}y_1^{2^mj}=y_1^{2^mj}$, which is impossible.
Hence we obtain that $(l,2)=1$.
If $t$ is divisible by $2$, then $1\neq y_2^{2^n}=x_1^{2^ns}y_1^{2^nt}=1$ since $n>m\geq 1$, a contradiction.
Hence $t$ is coprime to $2$.

If $m=1$, then $\langle x_1,y_2y_1^{-t}\rangle $ is isomorphic to $D_8$ (when $2|s$) or $Q_8$ (when $(s,2)=1$).
It follows that $$G_1\Y G_2=\langle y_2, x_1^ly_1^{2j_1}x_2^{-1} \rangle\Y\langle x_1,y_2y_1^{-t}\rangle\cong M_2(n+1,1,1)\Y D_8 \ or \ M_2(n+1,1,1)\Y Q_8. $$

Suppose that $m>1$. We may take $\varepsilon:=2^{m-1}$ (when $(s,2)=1$) and $0$ (when $2|s$). Hence
$y_2y_1^{-t}x_1^{\varepsilon-s}$ is of order $2$ and
$$G_1\Y G_2=\langle y_2, x_1^ly_1^{2j_1}x_2^{-1}\rangle\Y\langle x_1,y_2y_1^{-t}x_1^{\varepsilon-s}\rangle\cong M_2(n+1,1,1)\Y M_2(m+1,1). $$

(ii)-(v) may be obtained similar to Lemma \ref{GGC1}.

\end{proof}

\begin{theorem} \label{ST2}
Suppose that $n>m\geq 1$ and $c\in \langle z_2\rangle$. Then the isomorphism classes of $G$ are as follows.

(i) $M_2(m+1,n+1)\Y \underbrace{D_8\Y\cdots\Y D_8}_{k-1}\times\underbrace{\mathbb{Z}_2\times\cdots\times\mathbb{Z}_2}_{r-2}$, $k\geq 1$;

(ii) $M_2(m+1,n+1)\Y Q_8\Y\underbrace{D_8\Y\cdots\Y D_8}_{k-2} \times\underbrace{\mathbb{Z}_2\times\cdots\times\mathbb{Z}_2}_{r-2}$, $k\geq 2$;

(iii) $M_2(n+1,1,1)\Y M_2(m+1,1) \Y\underbrace{D_8\Y\cdots\Y D_8}_{k-2} \times\underbrace{\mathbb{Z}_2\times\cdots\times\mathbb{Z}_2}_{r-2}$, $k\geq 2$;

(iv) $M_2(n+1,1,1)\Y M_2(m+1,1)\Y Q_8 \Y\underbrace{D_8\Y\cdots\Y D_8}_{k-3}\Y \times\underbrace{\mathbb{Z}_2\times\cdots\times\mathbb{Z}_2}_{r-2}$, $k\geq 3$;

(v) $M_2(n+1,1,1)\Y\underbrace{D_8\Y\cdots\Y D_8}_{k-1}\Y \mathbb{Z}_{2^m}\times\underbrace{\mathbb{Z}_2\times\cdots\times\mathbb{Z}_2}_{r-2}$, $k\geq 1$;

(vi) $M_2(n+1,1,1)\Y Q_8 \Y\underbrace{D_8\Y\cdots\Y D_8}_{k-2}\Y \mathbb{Z}_{2^m} \times\underbrace{\mathbb{Z}_2\times\cdots\times\mathbb{Z}_2}_{r-2}$, $k\geq 2$;

(vii) $M_2(m+1,1)\Y \underbrace{D_8\Y\cdots\Y D_8}_{k-1}\times \mathbb{Z}_{2^{n}}\times\underbrace{\mathbb{Z}_2\times\cdots\times\mathbb{Z}_2}_{r-2}$, $k\geq 1$;

(viii) $M_2(m+1,1)\Y Q_8 \Y\underbrace{D_8\Y\cdots\Y D_8}_{k-2}\times \mathbb{Z}_{2^{n}}\times\underbrace{\mathbb{Z}_2\times\cdots\times\mathbb{Z}_2}_{r-2}$, $k\geq 2$;

(ix) $\underbrace{D_8\Y\cdots\Y D_8}_{k}\Y \mathbb{Z}_{2^{m}}\times \mathbb{Z}_{2^{n}}
\times\underbrace{\mathbb{Z}_2\times\cdots\times\mathbb{Z}_2}_{r-2}$, $k\geq 1$;

(x) $Q_8\Y\underbrace{D_8\Y\cdots\Y D_8}_{k-1}\Y \mathbb{Z}_{2^{m}}\times \mathbb{Z}_{2^{n}}
\times\underbrace{\mathbb{Z}_2\times\cdots\times\mathbb{Z}_2}_{r-2}$, $k\geq 1$.

\end{theorem}

\begin{proof}
According to Lemma \ref{STG}, we have that
$$G=E\Y A_1=G_1\Y G_2\Y\cdots\Y G_k\Y A_1=(G_1\Y N)\Y (G_2\Y N)\Y\cdots\Y (G_k\Y N)\Y A_1.$$
 According to Lemmas \ref{STG} and  \ref{CPC2},
  the isomorphism classes of $ G_i\Y N's$ are $M_2(m+1,n+1)$, $M_2(n+1,1,1)\Y \mathbb{Z}_{2^m}$, $M_2(m+1,1)\times \mathbb{Z}_{2^n}$, $D_8\Y \mathbb{Z}_{2^m} \times \mathbb{Z}_{2^n}$ and $Q_8\Y \mathbb{Z}_{2^m} \times \mathbb{Z}_{2^n}$.

If $\mathrm{Exp}(E)$ is equal to $2^{n+1}$, then there exists some $G_i$ such that $\mathrm{Exp}(G_i)=2^{n+1}$. Without loss of generality,
suppose $\mathrm{Exp}(G_1)=2^{n+1}$. Thus we have $G_1\Y N\cong M_2(m+1,n+1)$ or $M_2(n+1,1,1)\Y \mathbb{Z}_{2^{m}}$ by Lemma \ref{CPC2}.
If  the number $\mu$ of $ G_i\Y N's$ is odd, which  are isomorphic to $ M_2(m+1,n+1)$, then
we may obtain the isomorphism classes (i) and (ii) of $G$ by Lemma \ref{GGC2}. If $\lambda$ is a nonzero even number, then isomorphism classes of $G$ are (iii) and (iv). Suppose $\mu =0$. If  there exists some $G_i\Y N$ such that $G_i\Y N\cong M_2(m+1,1)\times \mathbb{Z}_{2^n}$, then  $M_2(n+1,1,1)\Y M_2(m+1,1)$ will produce $M_2(m+1,n+1)$, which implies $\mu\neq 0$. Obviously,
At the time, if  two factors $G_i$ and $G_j$ are isomorphic to $M_2(n+1,1,1)$, then $G_1\Y G_2$ is only isomorphic to $M_2(n+1,1,1)\Y D_8$ in (iv) of  Lemma \ref{GGC2}. It follows  the isomorphism classes (v) and (vi) of $G$.

Suppose that $\mathrm{Exp}(E)$ is less than $2^{n+1}$. At this case, the isomorphism classes of $(G_i\Y N)'s$ are $M_2(m+1,1)\times \mathbb{Z}_{2^n}$, $D_8\Y \mathbb{Z}_{2^m} \times \mathbb{Z}_{2^n}$ and $Q_8\Y \mathbb{Z}_{2^m} \times \mathbb{Z}_{2^n}$.
If there exists some $G_i\Y N$ such that $G_i\Y N\cong M_2(m+1,1)\times \mathbb{Z}_{2^n}$, then we may obtain isomorphism classes (vii)
and (viii) of $G$. Otherwise, (ix) and (x) may be obtained.
\end{proof}

 \subsection{ The isomorphism types $A_2$  and $A_3$ of $\zeta G$}

In the section,  we first consider
$$\zeta G=\langle z_1\rangle\times\langle z_2\rangle\times\cdots\times\langle z_r\rangle\cong\mathbb{Z}_{2^{n+1}}\times \mathbb{Z}_{2^{m}}\times\mathbb{Z}_2\times\cdots\times\mathbb{Z}_2, N=\langle z_1^2\rangle\times\langle z_2\rangle,$$
 Without loss of generality, we may always  let $c\in \langle z_1^2\rangle$ or $c\in \langle z_2\rangle$.
$E$ is the central product of $G_1,G_2,\ldots, G_k$, whose isomorphism classes  are  $Q_8$,$M_2(u,v)$ and $M_2(w,1,1)$, where $u,v,w\leq n+1$ in (2) of Lemma \ref{STG}. First we determine the types of $G_i\Y  \langle z_1,z_2\rangle$.

\begin{lemma} \label{CPC3}
Suppose $n\geq m\geq 1$ and $c\in \langle z_1\rangle$.

(i) If $G_i\cong M_2(n+1,v)$, then $1\leq v\leq m+1$ and $G_i\Y  \langle z_1,z_2\rangle\cong M_2(m+1,1,1)\Y \mathbb{Z}_{2^{n+1}}$ or $D_8\Y \mathbb{Z}_{2^{n+1}}\times \mathbb{Z}_{2^{m}}$.

(ii) If $n=m$ and $G_i\cong M_2(u, m+1)$, where $2\leq u<m+1$, then $G_i\Y  \langle z_1,z_2\rangle\cong M_2(m+1,1,1)\Y \mathbb{Z}_{2^{m+1}}$. If $n>m$, then $G_i$ is not isomorphic to $M_2(u,n+1)$, where $2\leq u\leq m+1$.

(iii) If $G_i\cong M_2(u,v)$, where $2\leq u<n+1$ and $1\leq v<n+1$, then $G_i\Y  \langle z_1,z_2\rangle\cong M_2(m+1,1,1)\Y \mathbb{Z}_{2^{n+1}}$ or $D_8\Y \mathbb{Z}_{2^{n+1}}\times \mathbb{Z}_{2^{m}}$.

(iv) If $G_i\cong M_2(w,1,1)$, then  $1\leq w<n+1$ (If $n>m$) or $1\leq w\leq m+1$ (If $n=m$).  Further, $G_i\Y  \langle z_1,z_2\rangle\cong M_2(m+1,1,1)\Y \mathbb{Z}_{2^{n+1}}$ or $D_8\Y \mathbb{Z}_{2^{n+1}}\times \mathbb{Z}_{2^{m}}$.

\end{lemma}

\begin{proof}
(i) Let
$$G_i=\langle x_i,y_i \ | \ x_i^{2^{n+1}}=y_i^{2^{v}}=1, x_i^{y_i}=x_i^{1+2^n} \rangle \cong M_2(n+1,v).$$
Since $\zeta G_i=\langle x_i^2, y_i^2\rangle\leq N$, we have $2^{n+v-1}\leq |N|=2^{n+m}$ and
$v\leq m+1$. Let $ x_i^2=z_1^{2l}z_2^{j}$ and $ y_i^2=z_1^{2s}z_2^{t}$, where
$1\leq l,s\leq 2^{n}, 1\leq j,t\leq 2^m$.

First, we suppose that $v=m+1$. Obviously, $\zeta G_i=N$ by comparing their orders.
If $t$ is divisible by $2$, then  $1\neq y_i^{2^m}=z_1^{2^{m}s}z_2^{2^{m-1}t}=z_1^{2^{m}s}$, which is impossible.
Hence $t$ and $2$ are coprime.

 We will next consider two cases $(j,2)=1$ and $2|j$. When $j$ and $2$ are coprime, we may suppose that $x_i^2=z_1^{2l}z_2$ and
$y_i^2=z_1^{2s}z_2$ without loss of generality. At this time, we have
$$G_i\Y\langle z_1,z_2\rangle=\langle y_iz_1^{-s}, x_iy_i^{-1}z_1^{s-l+2^{n-1}}\rangle\Y\langle z_1\rangle\cong M_2(m+1,1,1)\Y \mathbb{Z}_{2^{n+1}}.$$
When $j$ is divisible by $2$,  let $j=2j_1$, then we have
$$G_i\Y\langle z_1,z_2\rangle=\langle y_iz_1^{-s}, x_iz_1^{-l}z_2^{-j_1} \rangle\Y\langle z_1\rangle\cong M_2(m+1,1,1)\Y \mathbb{Z}_{2^{n+1}}.$$

We next suppose $1\leq v<m+1$. Since $1=y_i^{2^v}=z_1^{2^{v}s}z_2^{2^{v-1}t}$, we have
$t$ is divisible by  $2^{m-v+1}$. Let $t=2t_1$.
If $j$ and $2$ are coprime, then $$ G_i\Y\langle z_1,z_2\rangle=\langle x_iz_1^{-l},y_iz_1^{-s}z_2^{-t_1},z_1\rangle\cong M_2(m+1,1,1)\Y\mathbb{Z}_{2^{n+1}}.$$
Suppose that $j$ is divisible by $2$ and let $j=2j_1$. Hence
 $$ G_i\Y\langle z_1,z_2\rangle=\langle x_iz_1^{-l}z_2^{-j_1},y_iz_1^{-s}z_2^{-t_1}\rangle\Y\langle z_1,z_2 \rangle\cong D_8\Y\mathbb{Z}_{2^{n+1}}\times\mathbb{Z}_{2^m}.$$

(ii)  Let $$G_i=\langle x_i,y_i \ | \ x_i^{2^{u}}=y_i^{2^{n+1}}=1, x_i^{y_i}=x_i^{1+2^{u-1}} \rangle \cong M_2(u,n+1),$$
where  $2\leq u\leq m+1$. Obviously, $\zeta G_i=\langle x_i^2,y_i^2\rangle\leq\langle z_1^2,z_2\rangle$. At this time,  let $x_i^2=z_1^{2l}z_2^{j}$ and $y_i^2=z_1^{2s}z_2^{t}$.

If $n=m$, then we only consider the case  $2\leq u< m+1$ in order to avoid repeating (i).  Since $1= x_i^{2^u}=z_1^{2^ul}z_2^{2^{u-1}j}$,
both $l$ and $j$ are divisible by $2^{m+1-u}$.  Let $j=2j_1$.
Since $1\neq y_i^{2^m}=z_1^{2^ms}z_2^{2^{m-1}t}=c^sz_2^{2^{m-1}t}$,  we have $t$ and $2$ are coprime.
Hence
$$G_i\Y\langle z_1,z_2\rangle=\langle y_iz_1^{-s}, x_iz_1^{-l}z_2^{-j_1},z_1\rangle\cong M_2(m+1,1,1)\Y \mathbb{Z}_{2^{m+1}}.$$
Suppose that  $n>m$.
Note that $1\neq y_i^{2^n}=z_1^{2^ns}z_2^{2^{n-1}t}=z_1^{2^ns}=c^s$, a contradiction.

(iii)
Let
$$G_i=\langle x_i,y_i \ | \ x_i^{2^{u}}=y_i^{2^{v}}=1, x_i^{y_i}=x_i^{1+2^{u-1}} \rangle \cong M_2(u,v),$$
where $u,v\leq n$.
 Let $ x_i^2=z_1^{2l}z_2^{j}$ and $ y_i^2=z_1^{2s}z_2^{t}$, where
$1\leq l,s\leq 2^n, 1\leq j,t\leq 2^m$.
Obviously, both $l$ and $s$ are divisible by $2$. Let $l=2l_1$ and $s=2s_1$.

If $(j,2)=1=(t,2)$, then  we may let $ x_i^2=z_1^{2l}z_2$ and $ y_i^2=z_1^{2s}z_2$ without loss of generality. At this time, we have
$$G_i\Y\langle z_1,z_2\rangle=\langle x_iz_1^{-l}, x_iy_i^{-1}z_1^{s-l+2^{n-1}}, z_1\rangle\cong M_2(m+1,1,1)\Y \mathbb{Z}_{2^{n+1}}.$$

If  $(j,2)=1$ and $2|t$. Let $t=2t_1$, then we have
$$G_i\Y\langle z_1,z_2\rangle=\langle x_iz_1^{-l}, y_iz_1^{-s}z_2^{-t_1},z_1\rangle\cong M_2(m+1,1,1)\Y \mathbb{Z}_{2^{n+1}}.$$
Similarly, we may obtain the same result for the case $2|j$ and $(t,2)=1$.

Suppose that  $2|j$ and $2|t$. Let $j=2j_2$ and $t=2t_2$. We may obtain
$$G_i\Y\langle z_1,z_2\rangle=\langle x_iz_1^{-l}z_2^{-j_2}, y_iz_1^{-s}z_2^{-t_2},z_1,z_2\rangle\cong D_8\Y \mathbb{Z}_{2^{n+1}}\times \mathbb{Z}_{2^m}. $$

(iv)
Let
$$G_i=\langle x_i,y_i ,c\ | \ x_i^{2^{w}}=y_i^2=c^2=1, [x_i,y_i]=c,[x_i,c]=[y_i,c]=1 \rangle \cong M_2(w,1,1).$$
 Since $\langle x_i^2,c\rangle\leq \langle z_1^2,z_2\rangle$, we may let $ x_i^2=z_1^{2l}z_2^{j}$, where
$1\leq l\leq 2^n, 1\leq j\leq 2^m$. Obviously, we have $1\leq w\leq n+1$. If $w=n+1$ and $n>m$, then
$1\neq x_i^{2^n}=z_1^{2^nl}z_2^{2^{n-1}j}=z_1^{2^nl}=c^l$, which is impossible. Hence $w$ satisfies the following conditions:
$1\leq w\leq n+1$ (If $ n=m$) or $1\leq w< n+1$ (If $ n>m$).

If $j$ and $2$ are coprime, then we have
 $$G_i\Y\langle z_1,z_2\rangle=\langle x_iz_1^{-l}, y_i,z_1\rangle\cong M_2(m+1,1,1)\Y \mathbb{Z}_{2^{n+1}}. $$
 Suppose that $j$ is divisible by $2$ and let $j=2j_1$. It follows that
 $$G_i\Y\langle z_1,z_2\rangle=\langle x_iz_1^{-l}z_2^{-j_1}, y_i,z_1,z_2\rangle\cong D_8\Y \mathbb{Z}_{2^{n+1}}\times \mathbb{Z}_{2^{m}}. $$

\end{proof}

By Lemmas  \ref{STG} and \ref{CPC3},  the factor $G_i$ of the central product of $E$ may only be considered as three types:
$M_2(m+1,1,1), D_8, Q_8$. Note that  $G_1\Y G_2$ is only isomorphic to $M_2(m+1,1,1)\Y D_8$ if  both $G_1$
and $G_2$ are isomorphic to $M_2(m+1,1,1)$  according to  (iv) of Lemma \ref{GGC1}.  From this, we may obtain the following theorem
similar to Theorem \ref{ST1}.

\begin{theorem} \label{ST3}
Suppose that $n\geq m\geq 1$ and $c\in \langle z_1\rangle$. Then the isomorphism classes of $G$ are as follows.

(i) $M_2(m+1,1,1)\Y\underbrace{D_8\Y\cdots\Y D_8}_{k-1}\Y \mathbb{Z}_{2^{n+1}}\times\underbrace{\mathbb{Z}_2\times\cdots\times\mathbb{Z}_2}_{r-2}$, $k\geq 1$;

(ii) $M_2(m+1,1,1)\Y Q_8 \Y\underbrace{D_8\Y\cdots\Y D_8}_{k-2}\Y \mathbb{Z}_{2^{n+1}}\times\underbrace{\mathbb{Z}_2\times\cdots\times\mathbb{Z}_2}_{r-2}$, $k\geq 2$;

(iii) $\underbrace{D_8\Y\cdots\Y D_8}_{k}\Y \mathbb{Z}_{2^{n+1}}\times \mathbb{Z}_{2^{m}}
\times\underbrace{\mathbb{Z}_2\times\cdots\times\mathbb{Z}_2}_{r-2}$, $k\geq 1$;

(iv) $Q_8\Y\underbrace{D_8\Y\cdots\Y D_8}_{k-1}\Y \mathbb{Z}_{2^{n+1}}\times \mathbb{Z}_{2^{m}}
\times\underbrace{\mathbb{Z}_2\times\cdots\times\mathbb{Z}_2}_{r-2}$, $k\geq 1$.

\end{theorem}

\begin{lemma} \label{CPC4}
Suppose $n\geq m\geq 1$ and $c\in \langle z_2\rangle$.

(i) If $G_i\cong M_2(u,n+1)$, then $2\leq u\leq m+1$ and $$G_i\Y  \langle z_1,z_2\rangle\cong\left\{ \begin{aligned}
          & D_8\times \mathbb{Z}_{2^{n+1}}\ or\  Q_8\times \mathbb{Z}_{2^{n+1}}, \  m=1.\\
          & M_2(m+1,1)\times \mathbb{Z}_{2^{n+1}}\  or\  D_8\Y\mathbb{Z}_{2^m}\times\mathbb{Z}_{2^{n+1}}, \  m>1.
                                           \end{aligned} \right.$$

(ii) If $n=m$ and $G_i\cong M_2(m+1, v)$, where $1\leq v<m+1$, then $G_i\Y  \langle z_1,z_2\rangle\cong M_2(m+1,1)\times \mathbb{Z}_{2^{m+1}}$. If $n>m$, then $G_i$ is not isomorphic to $M_2(n+1,v)$, where $1\leq v\leq m+1$.

(iii) If $G_i\cong M_2(u,v)$, where $2\leq u<n+1$ and $1\leq v<n+1$, then  we may obtain the same results of $(i)$.

(iv) If $G_i\cong M_2(w,1,1)$, where $1\leq w\leq n+1$, then  $G_i\Y  \langle z_1,z_2\rangle\cong M_2(m+1,1)\times\mathbb{Z}_{2^{n+1}}$ or $D_8\Y  \mathbb{Z}_{2^{m}}\times\mathbb{Z}_{2^{n+1}}$.
\end{lemma}

\begin{proof}
(i) Let
$$G_i=\langle x_i,y_i \ | \ x_i^{2^{u}}=y_i^{2^{n+1}}=1, x_i^{y_i}=x_i^{1+2^{u-1}} \rangle \cong M_2(u,n+1).$$
Since $\zeta G_i=\langle x_i^2, y_i^2\rangle\leq N$, we have $2^{n+u-1}\leq |N|=2^{n+m}$ and
$2\leq u\leq m+1$. Let $ x_i^2=z_1^{2l}z_2^{j}$ and $ y_i^2=z_1^{2s}z_2^{t}$, where
$1\leq l,s\leq 2^{n}, 1\leq j,t\leq 2^m$.

First, we suppose that $u=m+1$. Obviously, $\zeta G_i=N$ by comparing their orders.
If $j$ is divisible by $2$, then  $c= x_i^{2^m}=z_1^{2^{m}l}z_2^{2^{m-1}j}=z_1^{2^{m}l}$, which is impossible.
Hence $j$ and $2$ are coprime.

 We will next consider two cases $(t,2)=1$ and $2|t$. When $t$ is divisible by $2$,  let $t=2t_1$, then
 $$G_i\Y\langle z_1,z_2\rangle=\langle x_iz_1^{-l}, y_iz_1^{-s}z_2^{-t_1}, z_1\rangle\cong M_2(m+1,1)\times \mathbb{Z}_{2^{n+1}}.$$
  When $t$ and $2$ are coprime, we may suppose that $x_i^2=z_1^{2l}z_2$ and
$y_i^2=z_1^{2s}z_2$ without loss of generality. If $m\geq 2$, then we have
$$G_i\Y\langle z_1,z_2\rangle=\langle x_iz_1^{-l}, x_iy_i^{-1}z_1^{s-l}z_2^{2^{m-2}}, z_1\rangle\cong M_2(m+1,1)\times \mathbb{Z}_{2^{n+1}}.$$
If $m=1$, then $u=2$ and $$G_i\Y\langle z_1,z_2\rangle=\langle x_i, y_iz_1^{-s}, z_1\rangle\cong Q_8\times \mathbb{Z}_{2^{n+1}}.$$

We next suppose $2\leq u<m+1$. Since $1=x_i^{2^u}=z_1^{2^{u}l}z_2^{2^{u-1}j}$, we have
$j$ is divisible by  $2^{m-u+1}$. Let $j=2j_1$.
If $t$ and $2$ are coprime, then $$ G_i\Y\langle z_1,z_2\rangle=\langle y_iz_1^{-s}, x_iz_1^{-l}z_2^{-j_1},z_1\rangle\cong M_2(m+1,1)\times\mathbb{Z}_{2^{n+1}}.$$
Suppose that $t$ is divisible by $2$ and let $t=2t_2$. Hence
 $$ G_i\Y\langle z_1,z_2\rangle=\langle x_iz_1^{-l}z_2^{-j_1},y_iz_1^{-s}z_2^{-t_2}\rangle\Y\langle z_1,z_2 \rangle\cong D_8\Y\mathbb{Z}_{2^m}\times\mathbb{Z}_{2^{n+1}}.$$

(ii)  Let $$G_i=\langle x_i,y_i \ | \ x_i^{2^{n+1}}=y_i^{2^{v}}=1, x_i^{y_i}=x_i^{1+2^{n}} \rangle \cong M_2(n+1,v),$$
where  $1\leq v\leq m+1$. Obviously, $\zeta G_i=\langle x_i^2,y_i^2\rangle\leq N=\langle z_1^2,z_2\rangle$. At this time,  let $x_i^2=z_1^{2l}z_2^{j}$ and $y_i^2=z_1^{2s}z_2^{t}$.

If $n=m$, then we only consider the case  $1\leq v< m+1$.  Since $c= x_i^{2^m}=z_1^{2^ml}z_2^{2^{m-1}j}=z_1^{2^ml}c^j$,
$j$ and $2$ are coprime.
Since $1= y_i^{2^v}=z_1^{2^{v}s}z_2^{2^{v-1}t}$,   $t$ are divisible by $2^{m-v+1}$, that is, $2|t$. Let $t=2t_1$.
Hence
$$G_i\Y\langle z_1,z_2\rangle=\langle x_iz_1^{-l}, y_iz_1^{-s}z_2^{-t_1},z_1\rangle\cong M_2(m+1,1)\times \mathbb{Z}_{2^{m+1}}.$$
Suppose that  $n>m$.
Note that $c= x_i^{2^n}=z_1^{2^nl}z_2^{2^{n-1}j}=z_1^{2^nl}$, a contradiction.

(iii)
Let
$$G_i=\langle x_i,y_i \ | \ x_i^{2^{u}}=y_i^{2^{v}}=1, x_i^{y_i}=x_i^{1+2^{u-1}} \rangle \cong M_2(u,v),$$
where $u,v\leq n$.
 Let $ x_i^2=z_1^{2l}z_2^{j}$ and $ y_i^2=z_1^{2s}z_2^{t}$, where
$1\leq l,s\leq 2^n, 1\leq j,t\leq 2^m$.

If  $(j,2)=1$ and $2|t$. Let $t=2t_1$, then we have
$$G_i\Y\langle z_1,z_2\rangle=\langle x_iz_1^{-l}, y_iz_1^{-s}z_2^{-t_1},z_1\rangle\cong M_2(m+1,1)\times \mathbb{Z}_{2^{n+1}}.$$
Similarly, we may obtain the same result for the case $2|j$ and $(t,2)=1$.

Suppose that  $2|j$ and $2|t$. Let $j=2j_2$ and $t=2t_2$. Then
$$G_i\Y\langle z_1,z_2\rangle=\langle x_iz_1^{-l}z_2^{-j_2}, y_iz_1^{-s}z_2^{-t_2},z_1,z_2\rangle\cong D_8\Y\mathbb{Z}_{2^m}\times \mathbb{Z}_{2^{n+1}}. $$

If $(j,2)=1=(t,2)$, then  we may let $ x_i^2=z_1^{2l}z_2$ and $ y_i^2=z_1^{2s}z_2$ without loss of generality. At this time, we
will distinguish the case $m=1$ from the case $m>1$. When $m=1$, we have
$$G_i\Y\langle z_1,z_2\rangle=\langle x_iz_1^{-l},y_iz_1^{-s}, z_1\rangle\cong Q_8\times\mathbb{Z}_{2^{n+1}}. $$
When $m>1$, then
$$G_i\ast
\langle z_1,z_2\rangle=\langle x_iz_1^{-l}, x_iy_i^{-1}z_1^{s-l}z_2^{2^{m-2}}, z_1\rangle\cong M_2(m+1,1)\times \mathbb{Z}_{2^{n+1}}.$$

(iv)
Let
$$G_i=\langle x_i,y_i ,c\ | \ x_i^{2^{w}}=y_i^2=c^2=1, [x_i,y_i]=c,[x_i,c]=[y_i,c]=1 \rangle \cong M_2(w,1,1).$$
Note that $\langle x_i^2,c\rangle\leq \langle z_1^2,z_2\rangle$. Thus we may let $ x_i^2=z_1^{2l}z_2^{j}$, where
$1\leq l\leq 2^n, 1\leq j\leq 2^m$.
If $j$ and $2$ are coprime, then
 $$G_i\Y\langle z_1,z_2\rangle=\langle x_iz_1^{-l}, y_i,z_1\rangle\cong M_2(m+1,1)\times \mathbb{Z}_{2^{n+1}}. $$
 Suppose that $j$ is divisible by $2$ and let $j=2j_1$. Then
 $$G_i\ast
 \langle z_1,z_2\rangle=\langle x_iz_1^{-l}z_2^{-j_1}, y_i,z_1,z_2\rangle\cong D_8\Y \mathbb{Z}_{2^{m}}\times\mathbb{Z}_{2^{n+1}} . $$

\end{proof}
By Lemmas  \ref{STG} and \ref{CPC4},  the factor $G_i$ of the central product of $E$ may only be considered as three types:
$M_2(m+1,1), D_8, Q_8$.
Note that  $G_1\Y G_2$ is only isomorphic to $M_2(m+1,1)\Y D_8$ if  both $G_1$
and $G_2$ are isomorphic to $M_2(m+1,1)$.  From this, we may obtain the following theorem
similar to Theorem \ref{ST2}.

\begin{theorem} \label{ST4}
Suppose that $n\geq m\geq 1$ and $c\in \langle z_2\rangle$. Then the isomorphism classes of $G$ are as follows.

(i) $M_2(m+1,1)\Y\underbrace{D_8\Y\cdots\Y D_8}_{k-1}\times \mathbb{Z}_{2^{n+1}}\times\underbrace{\mathbb{Z}_2\times\cdots\times\mathbb{Z}_2}_{r-2}$, $k\geq 1$;

(ii) $M_2(m+1,1)\Y Q_8 \Y\underbrace{D_8\Y\cdots\Y D_8}_{k-2}\times \mathbb{Z}_{2^{n+1}}\times\underbrace{\mathbb{Z}_2\times\cdots\times\mathbb{Z}_2}_{r-2}$, $k\geq 2$;

(iii) $\underbrace{D_8\Y\cdots\Y D_8}_{k}\Y  \mathbb{Z}_{2^{m}}\times\mathbb{Z}_{2^{n+1}}
\times\underbrace{\mathbb{Z}_2\times\cdots\times\mathbb{Z}_2}_{r-2}$, $k\geq 1$;

(iv) $Q_8\Y\underbrace{D_8\Y\cdots\Y D_8}_{k-1}\Y  \mathbb{Z}_{2^{m}}\times\mathbb{Z}_{2^{n+1}}
\times\underbrace{\mathbb{Z}_2\times\cdots\times\mathbb{Z}_2}_{r-2}$, $k\geq 1$.

\end{theorem}

Suppose that
$N=\langle z_1\rangle\times\langle z_2^2\rangle.$
 Similar to the case $A_2$, we have the following the results about $A_3$.

\begin{corollary} \label{ST5}
Suppose that $n\geq m\geq 1$ and $c\in \langle z_1\rangle$. Then the isomorphism classes of $G$ are as follows.

(i) $M_2(n+1,1)\Y\underbrace{D_8\Y\cdots\Y D_8}_{k-1}\times \mathbb{Z}_{2^{m+1}}\times\underbrace{\mathbb{Z}_2\times\cdots\times\mathbb{Z}_2}_{r-2}$, $k\geq 1$;

(ii) $M_2(n+1,1)\Y Q_8 \Y\underbrace{D_8\Y\cdots\Y D_8}_{k-2}\times \mathbb{Z}_{2^{m+1}}\times\underbrace{\mathbb{Z}_2\times\cdots\times\mathbb{Z}_2}_{r-2}$, $k\geq 2$;

(iii) $\underbrace{D_8\Y\cdots\Y D_8}_{k}\Y \mathbb{Z}_{2^{n}}\times \mathbb{Z}_{2^{m+1}}
\times\underbrace{\mathbb{Z}_2\times\cdots\times\mathbb{Z}_2}_{r-2}$, $k\geq 1$;

(iv) $Q_8\Y\underbrace{D_8\Y\cdots\Y D_8}_{k-1}\Y \mathbb{Z}_{2^{n}}\times \mathbb{Z}_{2^{m+1}}
\times\underbrace{\mathbb{Z}_2\times\cdots\times\mathbb{Z}_2}_{r-2}$, $k\geq 1$.

\end{corollary}

\begin{corollary} \label{ST6}
Suppose that $n\geq m\geq 1$ and $c\in \langle z_2\rangle$. Then the isomorphism classes of $G$ are as follows.

(i) $M_2(n+1,1,1)\Y\underbrace{D_8\Y\cdots\Y D_8}_{k-1}\Y \mathbb{Z}_{2^{m+1}}\times\underbrace{\mathbb{Z}_2\times\cdots\times\mathbb{Z}_2}_{r-2}$, $k\geq 1$;

(ii) $M_2(n+1,1,1)\Y Q_8 \Y\underbrace{D_8\Y\cdots\Y D_8}_{k-2}\Y \mathbb{Z}_{2^{m+1}}\times\underbrace{\mathbb{Z}_2\times\cdots\times\mathbb{Z}_2}_{r-2}$, $k\geq 2$;

(iii) $\underbrace{D_8\Y\cdots\Y D_8}_{k}\Y  \mathbb{Z}_{2^{m+1}}\times\mathbb{Z}_{2^{n}}
\times\underbrace{\mathbb{Z}_2\times\cdots\times\mathbb{Z}_2}_{r-2}$, $k\geq 1$;

(iv) $Q_8\Y\underbrace{D_8\Y\cdots\Y D_8}_{k-1}\Y  \mathbb{Z}_{2^{m+1}}\times\mathbb{Z}_{2^{n}}
\times\underbrace{\mathbb{Z}_2\times\cdots\times\mathbb{Z}_2}_{r-2}$, $k\geq 1$.

\end{corollary}

\subsection{ The isomorphism type $A_4$ of $\zeta G$}

In the section,
$$\zeta G=\langle z_1\rangle\times\langle z_2\rangle\times\cdots\times\langle z_r\rangle\cong\mathbb{Z}_{2^{n+1}}\times \mathbb{Z}_{2^{m+1}}\times\mathbb{Z}_2\times\cdots\times\mathbb{Z}_2, N=\langle z_1^2\rangle\times\langle z_2^2\rangle.$$
 Without loss of generality, we may always  let $c\in \langle z_1^2\rangle$ or $c\in \langle z_2^2\rangle$.
 First we determine the types of $G_i\Y  \langle z_1,z_2\rangle$, where $i=1,2,\ldots,k$.

\begin{lemma} \label{CPC7}
Suppose $n\geq m\geq 1$ and $c\in \langle z_1\rangle$ or $\langle z_2\rangle$.

(i) If $G_i\cong M_2(u,v)$, where $2\leq u\leq n+1$ and $1\leq v\leq n+1$, then $G_i\Y  \langle z_1,z_2\rangle\cong D_8\Y (\mathbb{Z}_{2^{n+1}}\times \mathbb{Z}_{2^{m+1}})$

(ii) If $G_i\cong M_2(w,1,1)$, where $1\leq w\leq n+1$,  then $G_i\Y  \langle z_1,z_2\rangle\cong D_8\Y (\mathbb{Z}_{2^{n+1}}\times \mathbb{Z}_{2^{m+1}})$.
\end{lemma}

\begin{proof}
(i) Let
$$G_i=\langle x_i,y_i \ | \ x_i^{2^{u}}=y_i^{2^{v}}=1, x_i^{y_i}=x_i^{1+2^{u-1}} \rangle \cong M_2(u,v).$$
Note that $\zeta G_i=\langle x_i^2, y_i^2\rangle\leq N$. Thus
we may let $ x_i^2=z_1^{2l}z_2^{2j}$ and $ y_i^2=z_1^{2s}z_2^{2t}$, where
$1\leq l,s\leq 2^{n}, 1\leq j,t\leq 2^m$. Hence
 $$G_i\Y\langle z_1,z_2\rangle=\langle x_iz_1^{-l}z_2^{-j},y_iz_1^{-s}z_2^{-t},z_1,z_2\rangle\cong D_8\Y (\mathbb{Z}_{2^{n+1}}\times \mathbb{Z}_{2^{m+1}}).$$

(ii)
Let
$$G_i=\langle x_i,y_i ,c\ | \ x_i^{2^{w}}=y_i^2=c^2=1, [x_i,y_i]=c,[x_i,c]=[y_i,c]=1 \rangle \cong M_2(w,1,1).$$
where $1\leq w\leq n+1$.
 Let $ x_i^2=z_1^{2l}z_2^{2j}$, where
$1\leq l\leq 2^n, 1\leq j\leq 2^m$. Thus
 $$G_i\Y\langle z_1,z_2\rangle=\langle x_iz_1^{-l}z_2^{-j},y_i,z_1,z_2\rangle\cong D_8\Y (\mathbb{Z}_{2^{n+1}}\times \mathbb{Z}_{2^{m+1}}).$$

\end{proof}

By Lemmas \ref{STG} and \ref{CPC7},  it is easy to obtain the following theorem since $Q_8\Y Q_8\cong Q_8\Y D_8$.

\begin{theorem} \label{ST7}
Suppose that $n\geq m\geq 1$ and $c\in \langle z_1\rangle$ or $\langle z_2\rangle$. Then the isomorphism classes of $G$ are as follows.

(i) $\underbrace{D_8\Y\cdots\Y D_8}_{k}\Y (\mathbb{Z}_{2^{n+1}}\times \mathbb{Z}_{2^{m+1}})
\times\underbrace{\mathbb{Z}_2\times\cdots\times\mathbb{Z}_2}_{r-2}$, $k\geq 1$;

(ii) $Q_8\Y\underbrace{D_8\Y\cdots\Y D_8}_{k-1}\Y (\mathbb{Z}_{2^{n+1}}\times \mathbb{Z}_{2^{m+1}})
\times\underbrace{\mathbb{Z}_2\times\cdots\times\mathbb{Z}_2}_{r-2}$, $k\geq 1$.

\end{theorem}

\section{ The unitary subgroup}

Let $G$ be  a nonabelian  $2$-group given by a
central extension of the form
$$1\longrightarrow N \longrightarrow G \longrightarrow \mathbb{Z}_2\times
\cdots\times \mathbb{Z}_2 \longrightarrow 1$$ and $G'=\langle c\rangle\cong
\mathbb{Z}_2$, $N\cong \mathbb{Z}_{2^n}\times \mathbb{Z}_{2^m}$,   $n\geq m\geq 1$.

 In section 3, we see there is an elementary abelian $2$-group in  direct product factors of $G$.
  According to Lemmas \ref{UDP} and \ref{SCG},
 we only consider the case $\zeta G=\langle z_1\rangle\times\langle z_2\rangle\geq N$ for
 computing the order of $U_{*}(FG)$.

Let $\bar{G}=G/G'$ and let $\Psi: FG\rightarrow F\overline{G}$ be the natural homomorphism.
Consider the sets
$$N_{\Psi}^*:=\{x\in V(FG)\ |\ \Psi(x)\in V_*(F\overline{G})\},$$
$$\mathrm{Ker}\Psi^+:=\{1+x\ |\ x\in \mathrm{Ker}\Psi\}.$$
It is easy to verify that $N_{\Psi}^*$ forms a subgroup of $V(FG)$ and $\mathrm{Ker}\Psi^+$ forms a normal subgroup of
$V(FG)$. Let $S_{G'}:=\{ xx^*\ |\ x\in N_{\Psi}^*\}$. We have that $S_{G'}\subseteq\mathrm{Ker}\Psi^+\leq N_{\Psi}^*$.

\begin{lemma}[\cite{Balogh}]\label{S1}
(i) $\mathrm{Supp}(xx^*)\cap \Omega_1(G)=\{1\}$ for every $x\in V(FG)$.

(ii) If $1+g\widehat{G'}\in S_{G'}$ for some $g\in G$, then $g^2=c$.

(iii) If $h^2$ is not included in $G'$ for any $h\in G$, then $1+\alpha(h+h^{-1})\widehat{G'}\in S_G'$ for any $\alpha\in F$.
\end{lemma}

We denote  by $\Theta(G)$ the order  of  the group
$\langle 1+\sum\limits_{g \in \Omega_c(G)}\alpha_gg\widehat{G'}\in S_{G'}, \alpha_g\in F\rangle$. From this,  we may
obtain the following lemma.

\begin{lemma}\label{S2}
$S_{G'}$ is a subgroup of $N_{\Psi}^*$.
And $$|V_*(FG)|=\frac{1}{\Theta(G)}|F|^{\frac{|G|+|\Omega_1(G)|+|\Omega_c(G)|}{4}}\cdot|V_*(F\overline{G})|.$$
In particular, if $1+\alpha g\widehat{G'}\in S_{G'}$ for any $g\in \Omega_c(G)$ and $\alpha\in F$, then
$$|V_*(FG)|=|F|^{\frac{|G|+|\Omega_1(G)|-|\Omega_c(G)|}{4}}\cdot|V_*(F\overline{G})|.$$
\end{lemma}

\begin{proof}
Note that the ideal $\mathrm{Ker}\Psi$ of $FG$ is generated by $1+c$. By Lemma  \ref{S1}, we may see
the elements of $S_{G'}$ are the following form
$$1+\sum_{g\in \Omega_c(G)}\alpha_gg\widehat{G'}+\sum_{h^2\notin G'}\beta_h(h+h^{-1})\widehat{G'}.$$
For any $1+\alpha x\widehat{G'}$,  $\sum\limits_{i=1}^{|G|}\alpha_ig_i\in V(FG)$,  where $\alpha, \alpha_i\in F$, $x, g_i\in G$, $i=1,2,\ldots, |G|$,  we have
\begin{align*}
\left(1+\alpha x\widehat{G'}\right)\left(\sum_{i=1}^{|G|}\alpha_ig_i\right) &=\sum_{i=1}^{|G|}\alpha_ig_i+\sum_{i=1}^{|G|}\alpha\alpha_ixg_i\widehat{G'}
=\sum_{i=1}^{|G|}\alpha_ig_i+\sum_{i=1}^{|G|}\alpha\alpha_ig_ix[x,g_i]\widehat{G'}\\
&=\sum_{i=1}^{|G|}\alpha_ig_i+\sum_{i=1}^{|G|}\alpha\alpha_ig_ix\widehat{G'}
=\left(\sum_{i=1}^{|G|}\alpha_ig_i\right)\left(1+\alpha x\widehat{G'}\right).
\end{align*}
Thus $S_{G'}$ is included  in  the center of $V(FG)$ and $S_{G'}$ is a subgroup of $N_{\Psi}^*$.
It is clear to see that
 $\langle 1+\sum_{h^2\notin G'}\beta_h(h+h^{-1})\widehat{G'}, \beta_h\in F\rangle$ has order $|F|^{\frac{|G|-|\Omega_1(G)|-|\Omega_c(G)|}{4}}$.
Hence $|S_{G'}|=\Theta(G)|F|^{\frac{|G|-|\Omega_1(G)|-|\Omega_c(G)|}{4}}$ according to the hypothesis.

Consider the mapping $\Phi:N_{\Psi}^*\rightarrow S_{G'}$. Obviously, $\Phi$ is an epimorphism.
Note that $\Psi(x)^{-1}=\Psi(x^{-1})=\Psi(x^*)=\Psi(x)^*$ if $x\in V_*(FG)$. Thus
$$\mathrm{Ker}(\Phi)=\{x\in N_{\Psi}^*\ |\ \Phi(x)=1\}=\{x\in V(FG)\ |\ x^*=x^{-1}\}=V_*(FG).$$
From this, we have $N_{\Psi}^*/V_*(FG)\cong S_{G'}$. Consider the restriction $\Psi|_{N_{\Psi}^*}:N_{\Psi}^*\rightarrow V_*(F\overline{G})$.
Obviously $\Psi|_{N_{\Psi}^*}$ is surjection and $\mathrm{Ker}(\Psi|_{N_{\Psi}^*})=\mathrm{Ker}\Psi^+$. It follows that
$$|V_*(FG)|=\frac{|N_{\Psi}^*|}{|S_{G'}|}=\frac{|\mathrm{Ker}\Psi^+|\cdot|V_*(F\overline{G})|}{|S_{G'}|}.$$
Since $$|\mathrm{Ker}\Psi^+|=|\mathrm{Ker}\Psi|=\frac{|FG|}{|F\overline{G}|}=F^{|G|/2},$$
we obtain $$|V_*(FG)|=\frac{|F|^{|G|/2}\cdot|V_*(F\overline{G})|}{|S_{G'}|}
=\frac{1}{\Theta(G)}|F|^{\frac{|G|+|\Omega_1(G)|+|\Omega_c(G)|}{4}}\cdot|V_*(F\overline{G})|.$$
In particular, if $1+\alpha g\widehat{G'}\in S_{G'}$ for any $g\in \Omega_c(G)$ and $\alpha\in F$, then $\Theta(G)=|F|^\frac{|\Omega_c(G)|}{2}$ and
$$|V_*(FG)|=|F|^{\frac{|G|+|\Omega_1(G)|-|\Omega_c(G)|}{4}}\cdot|V_*(F\overline{G})|.$$

\end{proof}

\begin{lemma} \label{INS}
Suppose that $h_1, h_2, h\in G$ and $\alpha \in F$. Then

(i) If  $h_1^4=h_2^4=1$, $h_1^2=h_2^2=c$ and $[h_1,h_2]=1$, then
$1+\alpha(h_1+h_2)\widehat{G'}\in S_{G'}$.

(ii) If $h_1^4= h_2^2=1$, $h_1^2=c$ and $[h_1,h_2]=1$, then
$1+\alpha h_1\widehat{G'}+\alpha^2 h_1h_2\widehat{G'}\in S_{G'}$.

(iii)  If  $h_1^2=h_2^2=1$ and $[h_1,h_2]=c$, then
$1+\alpha h_1h_2\widehat{G'}\in S_{G'}$.

(iv)  If  $h_1^4=h_2^2=1$ and $[h_1,h_2]=h_1^2=c$, then
$1+\alpha h_1\widehat{G'}\in S_{G'}$.

(v)  If  $h^8=1$ and $h^4=c$, then
$1+\alpha h^2\widehat{G'}\in S_{G'}$.

\end{lemma}
\begin{proof}
(i) follows since $(1+\alpha h_1+\alpha h_2)(1+\alpha h_1+\alpha h_2)^*=1+\alpha(h_1+h_2)\widehat{G'}$.

(ii) follows since $(1+\alpha h_1+\alpha h_2)(1+\alpha h_1+\alpha h_2)^*=1+\alpha h_1\widehat{G'}+\alpha^2 h_1h_2\widehat{G'}$.

(iii) follows since $(1+\alpha h_1h_2+\alpha h_2)(1+\alpha h_1h_2+\alpha h_2)^*=1+\alpha h_1h_2\widehat{G'}$.

(iv) follows since $(1+\alpha h_1+\alpha h_2)(1+\alpha h_1+\alpha h_2)^*=1+\alpha h_1\widehat{G'}$.

It is easy to verify $(\alpha+\alpha h^2+h)(\alpha+\alpha h^2+h)^*=1+\alpha^2h^2\widehat{G'}$. Note that the mapping:
$\phi:F\rightarrow F, \alpha\mapsto \alpha^2$ is an automorphism. From this, for any $\alpha\in F$,
we have $1+\alpha h^2\widehat{G'}\in S_{G'}$.

\end{proof}

For convenience,   we describe  the  relations of $D_8$ or $Q_8$  are as follows in the following part,
$$\langle x ,y \ | \  x^2=y^2=1, [x,y]=c\rangle\cong D_8,
 \langle x ,y  \ | \  x^4=1, x^2=y^2= [x,y]=c\rangle\cong Q_8.$$

\begin{lemma} \label{GINS}

(i)  Let $H_1=\langle a,b\ |\ a^4=b^4=1, [a,b]=a^2=c\rangle\cong \mathbb{Z}_4\rtimes \mathbb{Z}_4$.
Then $\Theta(H_1)$ is equal to $\frac{1}{2}|F|^{\frac{|\Omega_c(H_1)|}{2}}$.

(ii)  Let $H_2=\langle x,y\rangle\cong Q_8$.
Then $\Theta(H_2)$ is equal to  $\frac{1}{4}|F|^{\frac{|\Omega_c(H_2)|}{2}}$.
\end{lemma}
\begin {proof}
By Lemmas \ref{DQ}, \ref{CCU} and \ref{UDP}, we know $|V_*(FH_i)|=4|F|^{\frac{|H_i|+|\Omega_1(H_i)|}{2}-1}$.
Suppose that  $\overline{H_i}=H_i/H'_i$, where $i=1,2$,  then
$$|V_*(FH_i)|=\frac{1}{\Theta(H_i)}|F|^{\frac{|H_i|+|\Omega_1(H_i)|+|\Omega_c(H_i)|}{4}}\cdot|V_*(F\overline{H_i})|.$$
according to Lemma \ref{S2}.

For the group $H_1$,  $|\Omega_1(H_1)|=4=|\Omega_c(H_1)|$ and $\overline{H_1}\cong \mathbb{Z}_2\times\mathbb{Z}_4$.
Since
$$\frac{1}{\Theta(H_1)}|F|^{\frac{16+4+4}{4}}\cdot(2|F|^5)=4|F|^{\frac{16+4}{2}-1},$$ we have $\Theta(H_1)=\frac{|F|^2}{2}=\frac{1}{2}|F|^{\frac{|\Omega_c(H_1)|}{2}}$.

For the group $H_2$,  $|\Omega_1(H_2)|=2$, $|\Omega_c(H_2)|=6$ and $\overline{H_1}\cong \mathbb{Z}_2\times\mathbb{Z}_2$.
Since
$$\frac{1}{\Theta(H_2)}|F|^{\frac{8+2+6}{4}}\cdot(|F|^3)=4|F|^{\frac{8+2}{2}-1},$$ we have $\Theta(H_2)=\frac{|F|^3}{4}=\frac{1}{4}|F|^{\frac{|\Omega_c(H_2)|}{2}}$.

\end{proof}

\begin{lemma}\label{GINS2}
Suppose that  $H=\langle x,y \ | \ x^{4}=y^{2^{v}}=1, [x,y]=x^2=c\rangle\cong M_2(2, v)$, where  $v\geq 3$.
 Then $\Theta(H)=\frac{1}{2}|F|^{\frac{|\Omega_c(H)|}{2}}$.
\end{lemma}

\begin{proof}
It is easy to verify that $\Omega_c(H)=\{xy^ic^j \ |\ i=0, 2^{v-1}, j=0,1\}$.

Let $X:= \langle 1+\alpha x\widehat{H'}\in S_{H'}\ |\ \alpha\in F\rangle$.
First, we assert that  $X$  has order $\frac{1}{2}|F|$.
Since $(1+\beta x^2+\beta x)(1+\beta x^2+\beta x)^*=1+(\beta+\beta^2)x\widehat{G'}\in S_{G'}$ for any
$\beta\in F$,  the the order of $X$ is at less $\frac{1}{2}|F|$.
Every element of $FH$ can be written in the following form $a=a_1+a_2x+a_3y+a_4xy$, where
$a_i=a_{i1}+a_{i2}x^2$, $a_{ij}\in F\langle y^2\rangle$.
 $$\begin{aligned}
          aa^*&=a_1a_1^*+a_2a_2^*+a_3a_3^*+a_4a_4^*+(a_1^*a_2+a_3^*a_4)x+(a_1a_2^*+a_3a_4^*)x^3\\
          &+a_1a_3^*y^{-1}+a_1^*a_3y+a_1a_4^*y^{-1}x^{-1}+a_1^*a_4xy\\
          &+a_2a_3^*xy^{-1}+a_2^*a_3yx^{-1}+a_2a_4^*x^2y^{-1}+a_2^*a_4x^2y\\
          &=a_1a_1^*+a_2a_2^*+a_3a_3^*+a_4a_4^*+(a_1^*a_2+a_3^*a_4)x+(a_1a_2^*+a_3a_4^*)x^3\\
          &+b_1y^{-1}+b_2x^2y^{-1}+b_3y+b_4x^2y+b_5xy^{-1}+b_6x^3y^{-1}+b_7xy+b_8x^3y,
                                           \end{aligned}$$
where
$$\begin{aligned}
&b_1=a_{11}a_{31}^*+a_{12}a_{32}^*+a_{21}a_{42}^*+a_{22}a_{41}^*=b_3^*\\
&b_2=a_{11}a_{32}^*+a_{12}a_{31}^*+a_{21}a_{41}^*+a_{22}a_{42}^*=b_4^*\\
&b_5=a_{11}a_{41}^*+a_{12}a_{42}^*+a_{21}a_{31}^*+a_{22}a_{32}^*=b_7^*\\
&b_6=a_{11}a_{42}^*+a_{12}a_{41}^*+a_{21}a_{32}^*+a_{22}a_{31}^*=b_8^*
   \end{aligned}$$
Consider the augmentation mapping of
$FH$ to $F$, which is  denoted by $\chi$. Set $w_i=\chi(a_i)$ and $w_{ij}=\chi(a_{ij})$.
If $aa^*\in \in S_{H'}$ and $1+\alpha x\widehat{H'}$ in
$\langle 1+\sum\limits_{g \in \Omega_c(H)}\alpha_gg\widehat{H'}\in S_{H'}, \alpha_g\in F \rangle$ only  appears in the expression of $aa^*$,
then we have
$$\begin{aligned}
&\chi(a_1a_1^*+a_2a_2^*+a_3a_3^*+a_4a_4^*)=w_1^2+w_2^2+w_3^2+w_4^2=1\\
&\chi(a_1^*a_2+a_3^*a_4)=w_1w_2+w_3w_4=\alpha\\
&w_{11}w_{31}+w_{12}w_{32}+w_{21}w_{42}+w_{22}w_{41}=w_{11}w_{32}+w_{12}w_{31}+w_{21}w_{41}+w_{22}w_{42}\\
&w_{11}w_{41}+w_{12}w_{42}+w_{21}w_{31}+w_{22}w_{32}=w_{11}w_{42}+w_{12}w_{41}+w_{21}w_{32}+w_{22}w_{31}.
   \end{aligned}$$
From this,
$$w_1+w_2+w_3+w_4=1, w_1w_2+w_3w_4=\alpha, w_1w_3+w_2w_4=0, w_1w_4+w_2w_3=0.$$
According to Lemma 10 of \cite{Balogh}, the number of $\alpha$ is $\frac{1}{2}|F|$. Hence the assertion is true.

According to (i) of Lemma \ref{INS},  we have $1+\alpha(x+xy^{2^{v-1}})\widehat{G'}\in S_{G'}$ for any $\alpha \in F$.
From this, the order $\Theta(H)$ of $\langle 1+\sum\limits_{g \in \Omega_c(H)}\alpha_gg\widehat{H'}\in S_{H'}, \alpha_g\in F\rangle$
is $\frac{1}{2}|F|^2$, that is, $\Theta(H)=\frac{1}{2}|F|^{\frac{|\Omega_c(H)|}{2}}$.

\end{proof}

\begin{lemma}\label{USC}
Suppose that $G=H\Y K$, $H\bigcap K=\langle c\rangle$ and $H\cong M_2(u,v)$, where $u\geq 2$ and $v\geq 1$.
 If $(u,v)= (2,1)$ or $u>2$ ,  then $1+\alpha g\widehat{G'}\in S_{G'}$ for any $g\in \Omega_c(G)$, $\alpha\in F$.
\end{lemma}

\begin{proof}
Let
$$H=\langle x,y \ | \ x^{2^{u}}=y^{2^{v}}=1, x^{y}=x^{1+2^{u-1}}\rangle\cong M_2(u,v).$$
For any $g\in \Omega_c(G)$, suppose that $g=x^{i}y^{j}g_1$, where $0\leq i<2^{u-1}$, $0\leq j<2^v$ and $g_1\in K$.
Further we have $x^{2i+2^{u-1}ij}y^{2j}g_1^2=c$.

First, we consider  $u=2$ and $v=1$. At this time, $H\cong D_8$.
By taking $h_1=x$ and $h_2=y$ in (iv) of Lemma \ref{INS}, we have $1+\alpha x\widehat{G'}\in S_{G'}$ for any $\alpha \in F$.

If $g_1\in\Omega_c(G)$, then $x^{2i+2ij}=1$. Hence $(i,j)=(0,0),(0,1),(1,1)$.

$1+\alpha (x+g_1)\widehat{G'}\in S_{G'}$ follows by taking $h_1=x$ and $h_2=g_1$ in (i) of Lemma \ref{INS};

$1+\alpha g_1\widehat{G'}+\alpha^2 yg_1\widehat{G'}\in S_{G'}$ follows by taking $h_1=g_1$ and $h_2=y$ in (ii) of Lemma \ref{INS};

$1+\alpha xyg_1\widehat{G'}\in S_{G'}$ follows by taking $h_1=xg_1$ and $h_2=y$ in (iii) of Lemma \ref{INS}.

According to the above results, we have $1+\alpha g\widehat{G'}\in S_{G'}$ for  three cases about $(i,j)$.

If $g_1\in \Omega_1(G)$, then $x^{2i+2ij}=c$. Hence $i=1$ and $j=0$.
Further, by taking $h_1=x$ and $h_2=xg_1$ in (i) of Lemma \ref{INS}, we have $1+\alpha xg_1\widehat{G'}\in S_{G'}$ since $1+\alpha x\widehat{G'}\in S_{G'}$ for any $\alpha \in F$.

Suppose that $u$ is more than $2$. We have $1+\alpha x^{2^{u-2}}\widehat{G'}\in S_{G'}$ by taking $h=x^{2^{u-3}}$ in (v) of Lemma \ref{INS}.

If $g_1\in\Omega_c(G)$, then $x^{2i+2^{u-1}ij}y^{2j}=1$. Hence $j=0$ or $2^{v-1}$. From this,
we have $i=0$. Otherwise, $i=2^{u-2}$ and $x^{2^{u-1}(1+2^{u-2}j)}=1$, which is impossible.
By taking $h_1=x^{2^{u-2}}$ and $h_2=y^jg_1$ in  (i) of Lemma \ref{INS}, we have $1+\alpha y^jg_1\widehat{G'}\in S_{G'}$
since $1+\alpha x^{2^{u-2}}\widehat{G'}\in S_{G'}$.

If $g_1\in \Omega_1(G)$, then $x^{2i+2^{u-1}ij}y^{2j}=c$.  Hence $j=0$ or $2^{v-1}$. From this,
we have $i=2^{u-2}$.  For any $\alpha\in F$, by taking $h_1=x^{2^{u-2}}$ and $h_2=y^jg_1$ in (ii) of Lemma \ref{INS}, we have $1+\alpha x^{2^{u-2}}y^jg_1\widehat{G'}\in S_{G'}$.

\end{proof}

\begin{lemma}\label{USZ}
Suppose that $G=H\Y K$, $H\bigcap K=\langle c\rangle$ and $H\cong \mathbb{Z}_{2^u}$, where $u\geq 3$.
Then $1+\alpha g\widehat{G'}\in S_{G'}$ for any $g\in \Omega_c(G)$ and $\alpha\in F$.
\end{lemma}

\begin{proof}
Let $H=\langle z\rangle$. By the hypothesis, $z^{2^{u-1}}=c$.  $1+\alpha z^{2^{u-2}}\widehat{G'}\in S_{G'}$ follows
from (v) of Lemma \ref{INS}. For any $g\in \Omega_c(G)$, obviously $[z,g]=1$. $1+\alpha (z^{2^{u-2}}+g)\widehat{G'}\in S_{G'}$ follows from (i) of Lemma \ref{INS}, which implies $1+\alpha g\widehat{G'}\in S_{G'}$.

\end{proof}

\begin{lemma} \label{OQD}
$\Omega_1(D_8^{\Y k})=\Omega_c(Q_8\Y D_8^{\Y (k-1)})=2\gamma_1(k)$ and  $\Omega_c(D_8^{\Y k})=\Omega_1(Q_8\Y D_8^{\Y (k-1)})=2\gamma_2(k)$.
\end{lemma}

\begin{proof}
Let $$G_l=\langle x_1 ,y_1\rangle\Y\langle x_2,y_2\rangle\Y\cdots\Y\langle x_k,y_k\rangle\cong H_l\Y D_8^{\Y (k-1)} ,$$
where $\langle x_i ,y_i \rangle\cong D_8, i=2,\ldots,k,$ and $H_1\cong D_8$,  $H_2\cong Q_8$ , $l=1,2$.
For any $g\in G$, we may suppose that
$g=x_1^{i_1}y_1^{j_1}\cdots x_k^{i_k}y_k^{j_k}c^s,$
where $i_t, j_t, s=0$ or $1$,  $t=1,\ldots,k$.

For the group $G_1$,  we denote by $r_1$ the number  of set $\{(i_t,j_t)\ |\ i_t=j_t=1, t\in \{1,\ldots,k\}\}$.
If $g^2=1$, then  $r_1$ is even. At this time, we have $\Omega_1(G_1)=2\gamma_1(k)$.
If  $g^2=c$, then $r_1$ is odd. At this time, we have $\Omega_c(G_1)=2\gamma_2(k)$.

For the group $G_2$,  we denote by $r_2$ the number  of set $\{(i_t,j_t)\ |\ i_t=j_t=1, t\in \{2,\ldots,k\}\}$.
We first  suppose  $g^2=1$. If  $r_2$ is even, then $i_1=j_1=0$;  If  $r_2$ is odd, then $i_1+j_1\neq 0$.
At this time, we have
$$\Omega_1(G_2)=2\gamma_1(k-1)+6\gamma_2(k-1)=2^{2k}-2^{k}=2\gamma_2(k).$$
We next  suppose  $g^2=c$. If  $r_2$ is even, then $i_1+j_1\neq 0$;  If  $r_2$ is odd, then $i_1=j_1= 0$.
At this time, we have
$$\Omega_c(G_2)=6\gamma_1(k-1)+2\gamma_2(k-1)=2^{2k}+2^{k}=2\gamma_1(k).$$

\end{proof}

\begin{theorem} \label{USMD}
Suppose that $G\cong M_2(n+1,m+1)\Y D_8^{\Y (k-1)}$, where $n\geq m\geq 1$, $k\geq 1$. Then

(i) $|\Omega_1(G)|=|\Omega_c(G)|=2^{2k}$.

(ii) $|V_*(FG)|=\left\{ \begin{aligned}
          &2|F|^{\frac{|G|+|\Omega_1(G)|}{2}-1}, \ \mathrm{if}\   n=1\ \mathrm{and}\ k\geq 2.\\
          &4|F|^{\frac{|G|+|\Omega_1(G)|}{2}-1}, \ \mathrm{ otherwise}.
                                           \end{aligned} \right.$

\end{theorem}

\begin {proof}
(i) Let $G=\langle x_1 ,y_1\rangle\Y\langle x_2,y_2\rangle\Y\cdots\Y\langle x_k,y_k\rangle$,
where $\langle x_i ,y_i \rangle\cong D_8, i=2,\ldots,k,$ and
$$\langle x_1,y_1 \ | \ x_1^{2^{n+1}}=y_1^{2^{m+1}}=1, x_1^{y_1}=x_1^{1+2^{n}}\rangle\cong M_2(n+1,m+1).$$
For any $g\in G$, let $g=g_1g_2$, where
$g_1=x_1^{i_1}y_1^{j_1}$, $0\leq i_1< 2^n$, $0\leq j_1< 2^{m+1}$, and $g_2\in \langle x_2, y_2, \ldots, x_k, y_k\rangle$.

If $g_1^2=1$, then $x_1^{2i_1}y_1^{2j_1}c^{i_1j_1}=1$.
From this, we have $j_1=0$ or $2^{m}$ and $i_1=0$.
If $g_1^2=c$, then $x_1^{2i_1}y_1^{2j_1}c^{i_1j_1-1}=1$.
From this, we have $j_1=0$ or $2^{m}$ and $i_1=2^{n-1}$.
By Lemma \ref{OQD},  we have
$$|\Omega_1(G)|=4\gamma_1(k-1)+4\gamma_2(k-1)=2^{2k}=|\Omega_c(G)|.$$

(ii) We first consider the case $k=1$. At this time,
$G=\langle x_1 ,y_1\rangle.$
If $n=1$, then $m=1$ and $|V_{*}(FG)|=4 |F|^{\frac{|G|+|\Omega_1(G)|}{2}-1}$ from Lemma \ref{CCU}.

Suppose that $n\geq 2$.  By Lemma \ref{USC}, we may obtain
$1+\alpha g\widehat{G'}\in S_{G'}$ for any $g\in \Omega_c(G)$ and $\alpha\in F$.
By Lemma \ref{S2}, we have
$$|V_*(FG)|=|F|^{\frac{|G|+|\Omega_1(G)|-|\Omega_c(G)|}{4}}\cdot|V_*(F\overline{G})|=|F|^{2^{n+m}}\cdot|V_*(F\overline{G})|.$$
Note that $\overline{G}\cong \mathbb{Z}_{2^n}\times \mathbb{Z}_{2^{m+1}}$.
According to Lemma \ref{AU}, we have
$$|V_{*}(F\overline{G})|=|\overline{G}^2[2]|\cdot |F|^{\frac{|\overline{G}|+|\Omega_1(\overline{G})|}{2}-1}=4|F|^{2^{n+m}+1}.$$
Hence $|V_*(FG)|=4|F|^{2^{n+m+1}+1}$. Also since $\frac{|G|+|\Omega_1(G)|}{2}-1=\frac{2^{n+m+2}+4}{2}-1$ by (i),
$|V_*(FG)|=4|F|^{\frac{|G|+|\Omega_1(G)|}{2}-1}$.

We next suppose that $k\geq 2$. At this time, $D_8$ appears in the central product of $G$. By Lemma \ref{USC}, $1+\alpha g\widehat{G'}\in S_{G'}$ for any $g\in \Omega_c(G),  \alpha \in F$. Hence
$$|V_*(FG)|=|F|^{\frac{|G|+|\Omega_1(G)|-|\Omega_c(G)|}{4}}\cdot|V_*(F\overline{G})|=|F|^{2^{n+m+2k-2}}|V_*(F\overline{G})|$$
by Lemma \ref{S2}. Note that $\overline{G}\cong \mathbb{Z}_{2^{n}}\times\mathbb{Z}_{2^{m+1}}\times(\mathbb{Z}_2\times\mathbb{Z}_2)^{(k-1)} $.
Hence $ |V_*(F\overline{G})|=2|F|^{2^{n+m+2k-2}+2^{2k-1}-1}$ (If $n=1$) or $4|F|^{2^{n+m+2k-2}+2^{2k-1}-1}$ (If $n\geq 2$)  by Lemma \ref{AU}.  It follows that
$$|V_*(FG)|=\left\{ \begin{aligned}
          &2|F|^{\frac{|G|+|\Omega_1(G)|}{2}-1}, \ \mathrm{if}\   n=1\ \mathrm{and}\ k\geq 2.\\
          &4|F|^{\frac{|G|+|\Omega_1(G)|}{2}-1}, \ \mathrm{if}\   n\geq 2\ \mathrm{and}\ k\geq 2.
                                           \end{aligned} \right.$$

\end{proof}

\begin{theorem} \label{USDQ}
Suppose that $G\cong M_2(n+1,m+1)\Y Q_8\Y D_8^{\Y (k-2)}$, where $n\geq m\geq 1$, $k\geq 2$. Then

(i) $|\Omega_1(G)|=|\Omega_c(G)|=2^{2k}$.

(ii) $|V_*(FG)|=\left\{ \begin{aligned}
          &2|F|^{\frac{|G|+|\Omega_1(G)|}{2}-1}, \ \mathrm{if}\   n=1\ \mathrm{and}\ k\geq 3.\\
          &4|F|^{\frac{|G|+|\Omega_1(G)|}{2}-1}, \ \mathrm{ otherwise}.
                                           \end{aligned} \right.$

\end{theorem}

\begin {proof}
 Let $G=\langle x_1 ,y_1\rangle\Y\langle x_2,y_2\rangle\Y\cdots\Y\langle x_k,y_k\rangle$,
where $k\geq 2$, $\langle x_2 ,y_2 \rangle\cong Q_8$,
$\langle x_i ,y_i \rangle\cong D_8, i=3,\ldots,k,$ and
$$\langle x_1,y_1 \ | \ x_1^{2^{n+1}}=y_1^{2^{m+1}}=1, x_1^{y_1}=x_1^{1+2^{n}}\rangle\cong M_2(n+1,m+1).$$

(i) follows similar to (i) of Theorem \ref{USMD}.

(ii)  We will use the same notations in (i) and suppose that $g^2=c$.
We first consider the case $k=2$.  At this time,  the forms of $g$ are as follows:
$$x_1^{2^{n-1}}y_1^{j_1}c^s, y_1^{j_1}x_2c^s,y_1^{j_1}y_2c^s,y_1^{j_1}x_2y_2c^s.$$
If $n=1$, then $\langle x_1,y_1\rangle\cong \mathbb{Z}_4\rtimes\mathbb{Z}_4$.
By taking $h_1:=x_1$ and $h_2:=x_1y_1^2$ in (i) of  Lemma \ref{INS}, we have  $1+\alpha(x_1+x_1y_1^2)\widehat{G'}\in S_{G'}$ for
any $\alpha\in F$.
For any $\beta\in F$, it is easy to verify  $$(1+\beta x_1+\beta c)(1+\beta x_1+\beta c)^*=1+(\beta+\beta^2)x_1\widehat{G'}\in S_{G'}.$$
Hence $1+(\beta+\beta^2)x_1\widehat{G'}+\alpha(x_1+x_1y_1^2)\widehat{G'}\in S_{G'}$ for any $\alpha, \beta\in F$.
We know that the number
of $\gamma$ in $F$ which satisfies $1+\gamma x_1\widehat{G'}\in S_{G'}$ is $|F|/2$ by (i) of Lemma \ref{GINS}.
Suppose that $g$ is any other type  except for $x_1$. At this time, $g^4=x_1^4=1$, $g^2=x_1^2=c$ and $[x_1,g]=1$.
By (i) of Lemma \ref{INS}, we have $1+\alpha(x_1+g)\widehat{G'}\in S_{G'}$ for any $\alpha\in F$. In a word,
   $\langle 1+\sum_{g\in \Omega_c(G)}\alpha_gg\widehat{G'}\in S_{G'}, \alpha_g\in F \rangle$ has order $\frac{1}{2}|F|^{\frac{|\Omega_c(G)|}{2}}$. By Lemma \ref{S2},  $$|V_*(FG)|=\frac{2}{|F|^{\frac{|\Omega_c(G)|}{2}}}|F|^{\frac{|G|+|\Omega_1(G)|+|\Omega_c(G)|}{4}}\cdot|V_*(F\overline{G})|
=2|F|^{\frac{|G|+|\Omega_1(G)|-|\Omega_c(G)|}{4}}\cdot|V_*(F\overline{G})|.$$
Note that $\overline{G}\cong \mathbb{Z}_{2}\times \mathbb{Z}_4\times\mathbb{Z}_2\times\mathbb{Z}_2$.
According to Lemma \ref{AU}, we have
$$|V_{*}(F\overline{G})|=|\overline{G}^2[2]|\cdot |F|^{\frac{|\overline{G}|+|\Omega_1(\overline{G})|}{2}-1}=2|F|^{23}.$$
Hence $|V_*(FG)|=4|F|^{39}$. Also since $\frac{|G|+|\Omega_1(G)|}{2}-1=\frac{2^{6}+2^4}{2}-1=39$ by (i),
$|V_*(FG)|=4|F|^{\frac{|G|+|\Omega_1(G)|}{2}-1}$.

If $n\geq2$, then  $1+\alpha g\widehat{C}\in S_{G'}$  for any $g\in \Omega_c(G), \alpha\in F$ by Lemma \ref{USC}. From this,
by Lemma \ref{S2},  $$|V_*(FG)|=|F|^{\frac{|G|+|\Omega_1(G)|-|\Omega_c(G)|}{4}}\cdot|V_*(F\overline{G})|=|F|^{2^{n+m+2}}|V_*(F\overline{G})|
.$$ Note that $\overline{G}\cong \mathbb{Z}_{2^n}\times \mathbb{Z}_{2^{m+1}}\times\mathbb{Z}_2\times\mathbb{Z}_2$.
Thus $|V_*(F\overline{G})|=4|F|^{2^{n+m+2}+7}$. It follows that
$|V_*(FG)|=4|F|^{2^{n+m+3}+7}=4|F|^{\frac{|G|+|\Omega_1(G)|}{2}-1}$.

We next suppose that $k\geq 3$. At this time, $D_8$ appears in the central product of $G$. By Lemma \ref{USC}, $1+\alpha g\widehat{G'}\in S_{G'}$ for any $g\in\Omega_c(G),  \alpha \in F$.
Similar to the proof of Theorem \ref{USMD}, we may obtain (ii).

\end{proof}

\begin{theorem} \label{USMD2}
Suppose that $G_l\cong M_2(m+1,n+1)\Y H_l \Y D_8^{\Y (k-2)}$, where $n> m\geq 1$,
$H_1\cong D_8$ and $H_2\cong Q_8$, $l=1,2$ and $k\geq 2$. Then

(i) $|\Omega_1(G_l)|=|\Omega_c(G_l)|=2^{2k}$.

(ii) $|V_*(FG_l)|=\left\{ \begin{aligned}
          &2|F|^{\frac{|G_l|+|\Omega_1(G_l)|}{2}-1}, \ \mathrm{if}\ l=1= m, k\geq 2\  \mathrm{or}\  m=1, l=2, k\geq 3.\\
          &4|F|^{\frac{|G_l|+|\Omega_1(G_l)|}{2}-1}, \ otherwise.
                                           \end{aligned} \right.$
\end{theorem}

\begin {proof}
(i) Let $G_l=\langle x_1 ,y_1\rangle\Y\langle x_2,y_2\rangle\Y\cdots\Y\langle x_k,y_k\rangle$,
where $\langle x_2,y_2\rangle=H_l$, $\langle x_i ,y_i \rangle\cong D_8, i=3,\ldots,k,$ and
$$\langle x_1,y_1 \ | \ x_1^{2^{m+1}}=y_1^{2^{n+1}}=1, x_1^{y_1}=x_1^{1+2^{m}}\rangle\cong M_2(m+1,n+1).$$
For any $g\in G$, let $g=g_1g_2$, where
$g_1=x_1^{i_1}y_1^{j_1}$, $0\leq i_1< 2^m$, $0\leq j_1< 2^{n+1}$, and $g_2\in \langle x_2, y_2, \ldots, x_k, y_k\rangle$.

(i) follows similar to (i) of Theorem \ref{USMD}.

(ii)
First, we suppose that $m=1$.
 If $G_1\cong M_2(2,n+1)$. By Lemma \ref{GINS2}, we have $\Theta(G_1)=\frac{1}{2}|F|^{\frac{\Omega_c(G_1)}{2}}$.
Hence
$$|V_*(FG_1)|=2|F|^{\frac{|G_1|}{4}}\cdot|V_*(F\overline{G}_1)|=2|F|^{2^{1+n}}\cdot|V_*(F\overline{G}_1)|$$
 by Lemma \ref{S2}. Also since $\overline{G}_1\cong \mathbb{Z}_2\times\mathbb{Z}_{2^{n+1}}$,
 we have $$|V_*(FG_1)|=2|F|^{2^{1+n}}\cdot2 |F|^{2^{1+n}+1}=4 |F|^{2^{2+n}+1}=4 |F|^{\frac{|G_1|+|\Omega_1(G_1)|}{2}-1}.$$
If $G_2\cong M_2(2,n+1)\Y Q_8$, then we have $\Theta(G_2)=\frac{1}{2}|F|^{\frac{\Omega_c(G_2)}{2}}$ similar to Theorem \ref{USDQ}.
Note that $\overline{G}_2\cong \mathbb{Z}_2\times\mathbb{Z}_{2^{n+1}}\times \mathbb{Z}_2\times\mathbb{Z}_2$. Hence we have
$$|V_*(FG_2)|=2|F|^{2^{3+n}}\cdot|V_*(F\overline{G}_2)|=4|F|^{2^{4+n}+7}=4 |F|^{\frac{|G_2|+|\Omega_1(G_2)|}{2}-1}.$$

For other cases, by Lemma \ref{USC}, we may obtain
$1+\alpha g\widehat{G_l'}\in S_{G_l'}$ for any $g\in \Omega_c(G_l)$ and $\alpha\in F$.
 Hence
$|V_*(FG_l)|=|F|^{\frac{|G_l|}{4}}\cdot|V_*(F\overline{G}_l)|$
by Lemma \ref{S2}. Note that $\overline{G}_l\cong \mathbb{Z}_{2^{m}}\times\mathbb{Z}_{2^{n+1}}\times(\mathbb{Z}_2\times\mathbb{Z}_2)^{(k-1)} $.
Hence
$$|V_*(FG_l)|=\left\{ \begin{aligned}
          &2|F|^{2^{n+m+2k-1}+2^{2k-1}-1}=2|F|^{\frac{|G_l|+|\Omega_1(G_l)|}{2}-1}, \ \mathrm{if}\  m=1.\\
          &4|F|^{2^{n+m+2k-1}+2^{2k-1}-1}=4|F|^{\frac{|G_l|+|\Omega_1(G_l)|}{2}-1}, \ \mathrm{ if}\ m\geq 2.
                                           \end{aligned} \right.$$

\end{proof}

\begin{theorem} \label{UMMD}
Suppose that $G_l\cong M_2(n+1,1)\Y M_2(m+1,1,1)\Y H_l\Y D_8^{\Y (k-3)}$ ,
where $n\geq m\geq 1$, $k\geq 3$ and $l=1, 2$, $H_1\cong D_8$ and $H_2\cong Q_8$.

(i) If $n=1$, then
 $\Omega_1(G_1)=2^{2k}+2^{k+1}=\Omega_c(G_2)$
and  $\Omega_c(G_1)=2^{2k}-2^{k+1}=\Omega_1(G_2)$.
If $n\geq 2$,  then
 $\Omega_1(G_l)=2^{2k}=\Omega_c(G_l)$.

(ii) $|V_*(FG_l)|=\left\{ \begin{aligned}
          &2|F|^{\frac{|G_l|+|\Omega_1(G_l)|}{2}-1}, \ \mathrm{if}\   n=1.\\
          &4|F|^{\frac{|G_l|+|\Omega_1(G_l)|}{2}-1}, \ \mathrm{if}\   n\geq 2.
                                           \end{aligned} \right.$

\end{theorem}
\begin {proof}
(i) Let $G_l=\langle x_1 ,y_1\rangle\Y\langle x_2,y_2\rangle\Y\cdots\Y\langle x_k,y_k\rangle$,
where $\langle x_i ,y_i \rangle\cong D_8, i=4,\ldots,k$,
$\langle  x_3 ,y_3 \rangle=H_l$ and
$$\langle x_1,y_1 \ | \ x_1^{2^{n+1}}=y_1^{2}=1, x_1^{y_1}=x_1^{1+2^{n}}\rangle\cong M_2(n+1,1),$$
$$\langle x_2,y_2 \ | \ x_2^{2^{m+1}}=y_2^{2}=1, [x_2,y_2]=c\rangle\cong M_2(m+1,1,1).$$

For any $g\in G$, we may let
$g=g_1g_2$, where $g_1=x_1^{i_1}y_1^{j_1}x_2^{i_2}y_2^{j_2}$, $0\leq i_1< 2^n$, $0\leq i_2< 2^{m+1}$, $j_1, j_2=0$ or $1$ and $g_2\in \langle x_3,y_3,\ldots,x_k,y_k\rangle$.

If $g_1^2=1$, then $ x_1^{2i_1}x_2^{2i_2}c^{i_1j_1+i_2j_2}=1$. Hence  $i_2=0$ or $2^{m}$,  and $x_1^{2i_1+2^ni_1j_1}=1$. From this, we have $i_1=0$ if $j_1=0$,  $i_1=0$ or $2^{n-1}$ if $n=1$ and $j_1=1$,
$i_1=0$ if $n\geq 2$ and $j_1=1$.
If $g_1^2=c$, then $i_2=0$ or $2^{m}$,  and $x_1^{2i_1+2^n(i_1j_1-1)}=1$. From this, we have $i_1=2^{n-1}$ if $j_1=0$,  $i_1$ is non-value if $n=1$ and $j_1=1$,
$i_1=2^{n-1}$ if $n\geq 2$ and $j_1=1$.
 By Lemma \ref{OQD}, we have the following results:

(1) If $n=1$, then
we have  $\Omega_1(G_1)=24\gamma_1(k-2)+8\gamma_2(k-2)=2^{2k}+2^{k+1}=\Omega_c(G_2)$
and  $\Omega_c(G_1)=8\gamma_1(k-2)+24\gamma_2(k-2)=2^{2k}-2^{k+1}=\Omega_1(G_2)$.

(2) If $n\geq 2$,  then
we have  $\Omega_1(G_l)=16\gamma_1(k-2)+16\gamma_2(k-2)=2^{2k}=\Omega_c(G_l)$.

(ii)  According to Lemma \ref{USC}, we have  $1+\alpha g\widehat{G_l'}\in S_{G_l'}$ for any $g\in \Omega_c(G_l)$
and $\alpha \in F$. Hence $|V_*(FG_l)|=|F|^{\frac{|G_l|+|\Omega_1(G_l)|-|\Omega_c(G_l)|}{4}}\cdot|V_*(F\overline{G}_l)|$ by Lemma \ref{S2}.
 Note that $\overline{G}_l\cong \mathbb{Z}_{2^n}\times \mathbb{Z}_{2^{m+1}}\times\mathbb{Z}_2^{(2k-2)}$. From this, it follows that
$$|V_*(FG_l)|=\left\{ \begin{aligned}
          &2|F|^{2^{n+m+2k-1}+2^{2k-1}+\epsilon_l2^{k}-1}= 2|F|^{\frac{|G_l|+|\Omega_1(G_l)|}{2}-1},\ \mathrm{if}\   n=1.\\
          &4|F|^{2^{n+m+2k-1}+2^{2k-1}-1}=4|F|^{\frac{|G_l|+|\Omega_1(G_l)|}{2}-1}, \ \mathrm{if}\  n\geq 2.
                                           \end{aligned} \right.$$
where $\epsilon_1=1$ and $\epsilon_2=-1$.

\end{proof}

According to Theorem \ref{UMMD}, it is easy to obtain the following result.

\begin{corollary} \label{UMMD2}
Suppose that $G_l\cong M_2(m+1,1)\Y M_2(n+1,1,1)\Y H_l\Y D_8^{\Y (k-3)}$ ,
where $n> m\geq 1$, $k\geq 3$ and $l=1, 2$, $H_1\cong D_8$ and $H_2\cong Q_8$. Then
$|V_*(FG_l)|=\left\{ \begin{aligned}
          &2|F|^{\frac{|G_l|+|\Omega_1(G_l)|}{2}-1}, \ \mathrm{if}\   m=1.\\
          &4|F|^{\frac{|G_l|+|\Omega_1(G_l)|}{2}-1}, \ \mathrm{if}\   m\geq 2.
                                           \end{aligned} \right.$
\end{corollary}

\begin{theorem} \label{UM1D}
Suppose that $G_l\cong M_2(n+1,1)\Y H_l\Y D_8^{\Y (k-2)}\times\mathbb{Z}_{2^m}$ ,
where $n\geq m\geq 1$, $k\geq 2$ and $l=1, 2$, $H_1\cong D_8$ and $H_2\cong Q_8$. Then

(i) If $n=1$, $|\Omega_1(G_1)|=2^{2k+1}+2^{k+1}=|\Omega_c(G_2)|$ and $|\Omega_c(G_1)|=2^{2k+1}-2^{k+1}=|\Omega_1(G_2)|$;
If $n\geq 2$, $|\Omega_1(G_l)|=2^{2k+1}=|\Omega_c(G_l)|$, where $l=1,2$.

(ii) $|V_*(FG_l)|=\left\{ \begin{aligned}
          &|F|^{\frac{|G_l|+|\Omega_1(G_l)|}{2}-1}, \ \mathrm{if}\   n=m=1.\\
          &2|F|^{\frac{|G_l|+|\Omega_1(G_l)|}{2}-1}, \ \mathrm{if}\   n> m=1.\\
          &4|F|^{\frac{|G_l|+|\Omega_1(G_l)|}{2}-1},\   \mathrm{if}\   n\geq m\geq 2.
                                           \end{aligned} \right.$

\end{theorem}
\begin {proof}
(i) Let $G_l=\langle x_1 ,y_1\rangle\Y\langle x_2,y_2\rangle\Y\cdots\Y\langle x_k,y_k\rangle\times\langle z\rangle$,
 where  $\langle z\rangle\cong \mathbb{Z}_{2^{m}}$ and
$$\langle x_1,y_1 \ | \ x_1^{2^{n+1}}=y_1^{2}=1, x_1^{y_1}=x_1^{1+2^{n}}\rangle\cong M_2(n+1,1),$$
$$\langle  x_2 ,y_2  \rangle=H_l,l=1,2,\ \mathrm{and}\ \langle x_i ,y_i \rangle\cong D_8, i=3,\ldots,k.
$$
For any $g\in G$, let
$g=g_1g_2$, where $g_1=x_1^{i_1}y_1^{j_1}z^s$, $0\leq i_1< 2^n$, $0\leq j_1< 2$, $0\leq s< 2^m$ and $g_2\in \langle x_2,y_2,\ldots,x_k,y_k\rangle$.

If $g_1^2=1$, that is, $x_1^{2i_1}c^{i_1j_1}z^{2s}=1$, then $s=0$ or $2^{m-1}$. Furthermore, we have $(i_1,j_1)=(0,0),(0,1),(1,1)$ when $n=1$;
 $(i_1,j_1)=(0,0),(0,1)$ when $n\geq 2$.

If $g_1^2=c$, that is, $x_1^{2i_1}c^{i_1j_1}z^{2s}=c$, then $s=0$ or $2^{m-1}$. Furthermore, we have $(i_1,j_1)=(1,0)$ when $n=1$;
 $(i_1,j_1)=(2^{n-1},0),(2^{n-1},1)$ when $n\geq 2$.
 By Lemma \ref{OQD}, we have the following results:

(1) If $n=1$, then
we have  $\Omega_1(G_1)=12\gamma_1(k-1)+4\gamma_2(k-1)=2^{2k+1}+2^{k+1}=\Omega_c(G_2)$
and  $\Omega_c(G_1)=4\gamma_1(k-1)+12\gamma_2(k-1)=2^{2k+1}-2^{k+1}=\Omega_1(G_2)$.

(2) If $n\geq 2$,  then
we have  $\Omega_1(G_l)=8\gamma_1(k-1)+8\gamma_2(k-1)=2^{2k+1}=\Omega_c(G_l)$.

(ii)  According to Lemma \ref{USC},  we have  $1+\alpha g\widehat{G_l'}\in S_{G_l'}$ for any $g\in G_l$ and $\alpha \in F$. From this, by Lemma \ref{S2}, we obtain
$$|V_*(FG_l)|=|F|^{\frac{|G_l|+|\Omega_1(G_l)|-|\Omega_c(G_l)|}{4}}\cdot|V_*(F\overline{G}_l)|=\left\{ \begin{aligned}
          &|F|^{2^{n+m+2k-2}+2^{k}\varepsilon_l}\cdot|V_*(F\overline{G}_l)| , \ \mathrm{if}\  n=1.\\
          &|F|^{2^{n+m+2k-2}}\cdot|V_*(F\overline{G}_l)| , \ \mathrm{if}\   n\geq 2.
                                           \end{aligned} \right.$$
where $\varepsilon_1=1$ and $\varepsilon_2=-1$.
Note that $\overline{G}\cong \mathbb{Z}_{2^n}\times \mathbb{Z}_{2^{m}}\times\mathbb{Z}_2^{(2k-1)}$.
According to Lemma \ref{AU}, we have
$$|V_{*}(F\overline{G})=\left\{ \begin{aligned}
          &|F|^{2^{n+m+2k-2}+2^{2k}-1}, \ \mathrm{if}\    n=m=1.\\
          &2|F|^{2^{n+m+2k-2}+2^{2k}-1}, \ \mathrm{if}\   n> m=1.\\
                    &4|F|^{2^{n+m+2k-2}+2^{2k}-1}, \ \mathrm{if}\  n\geq m\geq 2.
                                           \end{aligned} \right.$$
From this, we have

$|V_*(FG_l)|=\left\{ \begin{aligned}
          &|F|^{2^{n+m+2k-1}+2^{2k}+2^k\varepsilon_l-1} =|F|^{\frac{|G_l|+|\Omega_1(G_l)|}{2}-1}, \ \mathrm{if}\  n=m=1.\\
                   &2|F|^{2^{n+m+2k-1}+2^{2k}-1} =2|F|^{\frac{|G_l|+|\Omega_1(G_l)|}{2}-1}, \ \mathrm{if}\   n>m=1.\\
          &4|F|^{2^{n+m+2k-1}+2^{2k}-1} =4|F|^{\frac{|G_l|+|\Omega_1(G_l)|}{2}-1},\ \mathrm{if}\  n\geq m\geq 2
                                           \end{aligned} \right.$.

\end{proof}

Obviously, by Lemma \ref{UDP}, we do not consider an elementary abelian $2$-group as a direct product term of $G$. According to Theorem \ref{UM1D}, the unitary subgroups of (i) and (ii) of Corollary \ref{ST5} are as follows.

\begin{corollary}\label{UM1DC}
Suppose that $G_l\cong M_2(n+1,1)\Y H_l\Y D_8^{\Y (k-2)}\times\mathbb{Z}_{2^{m+1}}$ ,
where $n\geq m\geq 1$, $k\geq 2$ and $l=1, 2$, $H_1\cong D_8$ and $H_2\cong Q_8$. Then

$$|V_*(FG_l)|=\left\{ \begin{aligned}
          &2|F|^{\frac{|G_l|+|\Omega_1(G_l)|}{2}-1}, \ \mathrm{if}\   n=m=1.\\
          &4|F|^{\frac{|G_l|+|\Omega_1(G_l)|}{2}-1}, \ \mathrm{if}\   n\geq 2.
                                                 \end{aligned} \right.$$
\end{corollary}

\begin{theorem} \label{UM1D2}
Suppose that $G_l\cong M_2(m+1,1)\Y H_l\Y D_8^{\Y (k-2)}\times\mathbb{Z}_{2^n}$ ,
where $n> m\geq 1$, $k\geq 2$ and $l=1, 2$, $H_1\cong D_8$ and $H_2\cong Q_8$. Then

(i) If $m=1$, $|\Omega_1(G_1)|=2^{2k+1}+2^{k+1}=|\Omega_c(G_2)|$ and $|\Omega_c(G_1)|=2^{2k+1}-2^{k+1}=|\Omega_1(G_2)|$;
If $m\geq 2$, $|\Omega_1(G_l)|=2^{2k+1}=|\Omega_c(G_l)|$.

(ii) $|V_*(FG_l)|=\left\{ \begin{aligned}
          &2|F|^{\frac{|G_l|+|\Omega_1(G_l)|}{2}-1}, \ \mathrm{if}\   n>m=1.\\
          &4|F|^{\frac{|G_l|+|\Omega_1(G_l)|}{2}-1}, \ \mathrm{if}\   n> m\geq 2.
                                                     \end{aligned} \right.$

\end{theorem}
\begin {proof}
(i) follows from (i) of Theorem \ref{UM1D}.

(ii)  According to Lemma \ref{USC},  we have  $1+\alpha g\widehat{G_l'}\in S_{G_l'}$ for any $g\in G_l$ and $\alpha \in F$. From this, by Lemma \ref{S2}, we obtain
$$|V_*(FG_l)|=|F|^{\frac{|G_l|+|\Omega_1(G_l)|-|\Omega_c(G_l)|}{4}}\cdot|V_*(F\overline{G}_l)|=\left\{ \begin{aligned}
          &|F|^{2^{n+m+2k-2}+2^{k}\varepsilon_l}\cdot|V_*(F\overline{G}_l)| , \ \mathrm{if}\  m=1.\\
          &|F|^{2^{n+m+2k-2}}\cdot|V_*(F\overline{G}_l)| , \ \mathrm{if}\   m\geq 2.
                                           \end{aligned} \right.$$
where $\varepsilon_1=1$ and $\varepsilon_2=-1$.
Note that $\overline{G}_l\cong \mathbb{Z}_{2^n}\times \mathbb{Z}_{2^{m}}\times\mathbb{Z}_2^{(2k-1)}$.
According to Lemma \ref{AU}, we have
$$|V_{*}(F\overline{G}_l)=\left\{ \begin{aligned}
          &2|F|^{2^{n+m+2k-2}+2^{2k}-1}, \ \mathrm{if}\    m=1.\\
          &4|F|^{2^{n+m+2k-2}+2^{2k}-1}, \ \mathrm{if}\   m\geq 2.
                                                              \end{aligned} \right.$$
From this, we have

$|V_*(FG_l)|=\left\{ \begin{aligned}
          &2|F|^{2^{n+m+2k-1}+2^{2k}+2^k\varepsilon_l-1} =2|F|^{\frac{|G_l|+|\Omega_1(G_l)|}{2}-1}, \ \mathrm{if}\  m=1.\\
          &4|F|^{2^{n+m+2k-1}+2^{2k}-1} =4|F|^{\frac{|G_l|+|\Omega_1(G_l)|}{2}-1},\ \mathrm{if}\  m\geq 2.
                                           \end{aligned} \right.$

\end{proof}

Obviously, by Lemma \ref{UDP}, we do not consider an elementary abelian $2$-group as a direct product term of $G$. Furthermore, the unitary subgroups of (i) and (ii) of Theorem \ref{ST4} are the special cases of Theorem \ref{UM1D2}.

\begin{theorem} \label{M11DQ}
Suppose that $G_l\cong M_2(m+1,1,1)\Y H_l\Y D_8^{\Y (k-2)}\Y\mathbb{Z}_{2^n}$ ,
where $n\geq m\geq 1$, $k\geq 2$ and $l=1, 2$, $H_1\cong D_8$ and $H_2\cong Q_8$. Then

(i) If $n=1$, $\Omega_1(G_1)=2^{2k}+2^{k+1}=\Omega_c(G_2)$
and  $\Omega_c(G_1)=2^{2k}-2^{k+1}=\Omega_1(G_2)$;
If $n\geq 2$, $\Omega_1(G_l)=2^{2k+1}=\Omega_c(G_l)$.

(ii) $|V_*(FG_1)|=\left\{ \begin{aligned}
          &2|F|^{\frac{|G_1|+|\Omega_1(G_1)|}{2}-1}, \ \mathrm{if}\    n=1\ \mathrm{or}\ n=2 \ \mathrm{and}\ k\geq 2.\\
          &4|F|^{\frac{|G_1|+|\Omega_1(G_1)|}{2}-1}, \   \mathrm{ otherwise}.
                                          \end{aligned} \right.$

$|V_*(FG_2)|=\left\{ \begin{aligned}
           &8|F|^{\frac{|G_2|+|\Omega_1(G_2)|}{2}-1}, \ \mathrm{if}\    n=1\ \mathrm{and}\ k=2.\\
          &2|F|^{\frac{|G_2|+|\Omega_1(G_2)|}{2}-1}, \ \mathrm{if}\    1\leq n\leq 2\ \mathrm{and}\ k\geq 3.\\
          &4|F|^{\frac{|G_2|+|\Omega_1(G_2)|}{2}-1}, \ \mathrm{ otherwise}.
                                           \end{aligned} \right.$
\end{theorem}

\begin {proof}
 Let $G_l=\langle x_1 ,y_1\rangle\Y\langle x_2,y_2\rangle\Y\cdots\Y\langle x_k,y_k\rangle\Y\langle z\rangle$,
where $k\geq 2$, $\langle x_2 ,y_2 \rangle=H_l$,
$\langle x_i ,y_i \rangle\cong D_8, i=3,\ldots,k$, $\langle z\rangle\cong\mathbb{Z}_{2^n}$ and
$$\langle x_1,y_1 \ | \ x_1^{2^{m+1}}=y_1^{2}=c^2=1, [x_1,y_1]=c\rangle\cong M_2(m+1,1,1).$$

(i) For any $g\in G$, we may let
$g=g_1g_2$, where $g_1=x_1^{i_1}y_1^{j_1}z^s$, $0\leq i_1< 2^{m+1}$, $0\leq j_1< 2$, $0\leq s< 2^{n-1}$ and $g_2\in \langle x_2,y_2,\ldots,x_k,y_k\rangle$.
If $g_1^2=1$, then $x_1^{2i_1}y_1^{2j_1}z^{2s+2^{n-1}i_1j_1}=1$. Hence $i_1=0$ or $2^m$, $j_1=0$ or $1$, and
$s=0$.
If $g_1^2=c$, then $n\geq 2$ and $x_1^{2i_1}y_1^{2j_1}z^{2s+2^{n-1}(i_1j_1-1)}=1$. Hence $i_1=0$ or $2^m$, $j_1=0$ or $1$, and $s=2^{n-2}$.

(1) If $n=1$, then
we have  $\Omega_1(G_1)=8\gamma_1(k-1)=2^{2k}+2^{k+1}=\Omega_c(G_2)$
and  $\Omega_c(G_1)=8\gamma_2(k-1)=2^{2k}-2^{k+1}=\Omega_1(G_2)$.

(2) If $n\geq 2$,  then
we have  $\Omega_1(G_l)=8\gamma_1(k-1)+8\gamma_2(k-1)=2^{2k+1}=\Omega_c(G_l)$.

(ii) First, we consider the case $n=1$. If $G_1\cong M_2(2,1,1)$, then $\Omega_c(G_1)=\emptyset$.
Hence $\Theta(G_1)=1$ in Lemma \ref{S2}. Furthermore,  $$|V_*(FG_1)|=|F|^{\frac{|G_1|+|\Omega_1(G_1)|+|\Omega_c(G_1)|}{4}}\cdot|V_*(F\overline{G}_1)|.$$
Note that $\overline{G}_1\cong \mathbb{Z}_{2^{2}}\times\mathbb{Z}_{2}$.
It follows that $|V_*(FG_1)|=2|F|^{11}=2|F|^{\frac{|G_1|+|\Omega_1(G_1)|}{2}-1}$.

Suppose that $G_2\cong M_2(2,1,1)\Y Q_8$. Note that
the group $\langle 1+\sum\limits_{g \in \Omega_c(G_2)}\alpha_gg\widehat{G_2'}\in S_{G_2'}, \alpha_g\in F\rangle$
is generated by the following elements
$$1+(\alpha^2+\alpha)x_2\widehat{G_2'},1+(\alpha^2+\alpha)y_2\widehat{G_2'},1+\alpha(x_2+y_2+x_2y_2)\widehat{G_2'},$$ $$1+\alpha(x_2+x_2g)\widehat{G_2'},1+\alpha(y_2+y_2g)\widehat{G_2'},1+\alpha(x_2y_2+x_2y_2g)\widehat{G_2'},$$
where $\alpha\in F, g\in \{ x_1^{2}, y_1,  y_1x_1^{2}\}$,
according to (i) of Lemma \ref{INS} and (ii) of Lemma \ref{GINS}. Hence
$$\Theta(G_2)=\frac{1}{4}|F|^{12}=\frac{1}{4}|F|^{\frac{\Omega_c(G_2)}{2}}.$$
By Lemma \ref{S2},  $$|V_*(FG_2)|=\frac{4}{|F|^{\frac{|\Omega_c(G_2)|}{2}}}
|F|^{\frac{|G_2|+|\Omega_1(G_2)|+|\Omega_c(G_2)|}{4}}\cdot|V_*(F\overline{G}_2)|
=4|F|^{\frac{|G_2|+|\Omega_1(G_2)|-|\Omega_c(G_2)|}{4}}\cdot|V_*(F\overline{G}_2)|.$$
Note that $\overline{G}_2\cong \mathbb{Z}_{2^{2}}\times \mathbb{Z}_2\times\mathbb{Z}_2\times\mathbb{Z}_2$.
According to Lemma \ref{AU}, we have
$$|V_{*}(F\overline{G}_2)|=|\overline{G}_2^2[2]|\cdot |F|^{\frac{|\overline{G}_2|+|\Omega_1(\overline{G}_2)|}{2}-1}=2|F|^{23}.$$
Hence $|V_*(FG_2)|=8|F|^{35}=8|F|^{\frac{|G_2|+|\Omega_1(G_2)|}{2}-1}$.

For the other cases when $n=1$, $D_8$ will appear in the central product of $G_l$, where $l=1, 2$.
At this time, we have $1+\alpha g\widehat{G_l'}\in S_{G_l'}$ by Lemma \ref{USC}.
From this,
by Lemma \ref{S2},  $$|V_*(FG_l)|=|F|^{\frac{|G_l|+|\Omega_1(G_l)|-|\Omega_c(G_l)|}{4}}\cdot|V_*(F\overline{G}_l)|
.$$
Note that $\overline{G}_l\cong \mathbb{Z}_{2^{m+1}}\times \mathbb{Z}_{2}^{(2k-1)}$. Hence
$$|V_*(FG_1)|=2|F|^{2^{m+2k}+2^{2k-1}+2^k-1}=2|F|^{\frac{|G_1|+|\Omega_1(G_1)|}{2}-1}.$$
$$|V_*(FG_2)|=2|F|^{2^{m+2k}+2^{2k-1}-2^k-1}=2|F|^{\frac{|G_2|+|\Omega_1(G_2)|}{2}-1}.$$

Suppose that $n=2$ and $G_1\cong M_2(m+1,1,1)\Y\mathbb{Z}_4$ or $G_2\cong M_2(m+1,1,1)\Y Q_8\Y\mathbb{Z}_4$.

We have $|V_*(F\langle z\rangle)|=2|F|^2$ by Lemma \ref{AU}. From this,
$\Theta\langle z\rangle=\frac{1}{2}|F|$ according to Lemma \ref{S2}. Note that
$1+(\beta^2+\beta)z\widehat{G_l'}\in S_{G_l'}$
since $(1+\beta z+\beta c)(1+\beta z+\beta c)^*\in S_{G_l'}$ for any $\beta\in F$, where $l=1,2$,
which implies $\langle1+\alpha z\widehat{G_l'}\in S_{G_l'}, \alpha\in F\rangle$ has exact $\frac{1}{2}|F|$ elements.
For any $g\in \Omega_c(G_l)$,  we have $1+\alpha(z+g)\widehat{G_l'}\in S_{G_l'}$ for any $\alpha\in F$ by (i) of Lemma \ref{INS}. It follows that $\Theta(G_l)=\frac{1}{2}|F|^{\frac{\Omega_c(G_l)}{2}}$, where $l=1,2$.
By Lemma \ref{S2},  $$|V_*(FG_l)|=2|F|^{\frac{|G_l|+|\Omega_1(G_l)|-|\Omega_c(G_l)|}{4}}\cdot|V_*(F\overline{G}_l)|.$$
Note that $\overline{G}_1\cong \mathbb{Z}_{2^{m+1}}\times \mathbb{Z}_{2}^{(2)}$
and $\overline{G}_2\cong \mathbb{Z}_{2^{m+1}}\times \mathbb{Z}_{2}^{(4)}$. Hence
$$|V_*(FG_1)|=4|F|^{2^{m+3}+3}=4|F|^{\frac{|G_1|+|\Omega_1(G_1)|}{2}-1}.$$
$$|V_*(FG_2)|=4|F|^{2^{m+5}+15}=4|F|^{\frac{|G_2|+|\Omega_1(G_2)|}{2}-1}.$$
For the other cases when $n=2$, $D_8$ will appear in the central product of $G_l$, where $l=1, 2$.
At this time, we have $1+\alpha g\widehat{G_l'}\in S_{G_l'}$ for any $\alpha\in F$ and $g\in \Omega_c(G_l)$ by Lemma \ref{USC}.
From this,
by Lemma \ref{S2} and (i),  $$|V_*(FG_l)|=|F|^{\frac{|G_l|}{4}}\cdot|V_*(F\overline{G}_l)|.$$
Note that $\overline{G}_l\cong \mathbb{Z}_{2^{m+1}}\times \mathbb{Z}_{2}^{(2k)}$. Hence
$$|V_*(FG_l)|=2|F|^{2^{m+2k+1}+2^{2k}-1}=2|F|^{\frac{|G_l|+|\Omega_1(G_l)|}{2}-1}.$$

If $n\geq3$, then  $1+\alpha g\widehat{G_l'}\in S_{G_l'}$ for any $\alpha\in F$ and $g\in \Omega_c(G_l)$ by Lemma \ref{USZ}. From this,
by Lemma \ref{S2},  $$|V_*(FG_l)|=|F|^{\frac{|G_l|}{4}}\cdot|V_*(F\overline{G}_l)|.$$
Note that $\overline{G}_l\cong \mathbb{Z}_{2^{m+1}}\times\mathbb{Z}_2^{(2k-1)}\times\mathbb{Z}_{2^{n-1}}$.
It follows that
$$|V_*(FG_l)|=4|F|^{2^{n+m+2k-1}+2^{2k}-1}=4|F|^{\frac{|G_l|+|\Omega_1(G_l)|}{2}-1}.$$

\end{proof}

Obviously, by Lemma \ref{UDP}, the unitary subgroups of (i) and (ii) of Theorem \ref{ST3} are the special cases of Theorem \ref{M11DQ}. Furthermore, it is easy to obtain the following results similar to Theorem \ref{M11DQ}.

\begin{corollary} \label{M11DQ2}
Suppose that $G_l\cong M_2(n+1,1,1)\Y H_l\Y D_8^{\Y (k-2)}\Y\mathbb{Z}_{2^m}$ ,
where $n> m\geq 1$, $k\geq 2$ and $l=1, 2$, $H_1\cong D_8$ and $H_2\cong Q_8$. Then

$|V_*(FG_1)|=\left\{ \begin{aligned}
          &2|F|^{\frac{|G_1|+|\Omega_1(G_1)|}{2}-1}, \ \mathrm{if}\    m=1\ \mathrm{or}\ m=2 \ \mathrm{and}\ k\geq 2.\\
          &4|F|^{\frac{|G_1|+|\Omega_1(G_1)|}{2}-1}, \   \mathrm{ otherwise}.
                                          \end{aligned} \right.$

$|V_*(FG_2)|=\left\{ \begin{aligned}
           &8|F|^{\frac{|G_2|+|\Omega_1(G_2)|}{2}-1}, \ \mathrm{if}\    m=1\ \mathrm{and}\ k=2.\\
          &2|F|^{\frac{|G_2|+|\Omega_1(G_2)|}{2}-1}, \ \mathrm{if}\    1\leq m\leq 2\ \mathrm{and}\ k\geq 3.\\
          &4|F|^{\frac{|G_2|+|\Omega_1(G_2)|}{2}-1}, \ \mathrm{ otherwise}.
                                           \end{aligned} \right.$
\end{corollary}

\begin{corollary} \label{M11DQ3}
Suppose that $G_l\cong M_2(n+1,1,1)\Y H_l\Y D_8^{\Y (k-2)}\Y\mathbb{Z}_{2^{m+1}}$ ,
where $n\geq m\geq 1$, $k\geq 2$ and $l=1, 2$, $H_1\cong D_8$ and $H_2\cong Q_8$. Then

$|V_*(FG_1)|=\left\{ \begin{aligned}
          &2|F|^{\frac{|G_1|+|\Omega_1(G_1)|}{2}-1}, \ \mathrm{if}\    m=1 \ \mathrm{and}\ k\geq 2.\\
          &4|F|^{\frac{|G_1|+|\Omega_1(G_1)|}{2}-1}, \   \mathrm{ otherwise}.
                                          \end{aligned} \right.$

$|V_*(FG_2)|=\left\{ \begin{aligned}
                     &2|F|^{\frac{|G_2|+|\Omega_1(G_2)|}{2}-1}, \ \mathrm{if}\    m=1\ \mathrm{and}\ k\geq 3.\\
          &4|F|^{\frac{|G_2|+|\Omega_1(G_2)|}{2}-1}, \ \mathrm{ otherwise}.
                                           \end{aligned} \right.$
\end{corollary}

\begin{theorem} \label{DQZZ}
Suppose that $G_l\cong  H_l\Y D_8^{\Y (k-1)}\Y\mathbb{Z}_{2^{n_1}}\times\mathbb{Z}_{2^{m_1}}$ ,
where $n_1, m_1, k\geq 1$  and  $l=1, 2$, $H_1\cong D_8$,  $H_2\cong Q_8$. Then

(i) If $n_1=1$, $\Omega_1(G_1)=2^{2k+1}+2^{k+1}=\Omega_c(G_2)$
and  $\Omega_c(G_1)=2^{2k+1}-2^{k+1}=\Omega_1(G_2)$;
If $n_1\geq 2$, $\Omega_1(G_l)=2^{2k+2}=\Omega_c(G_l)$, where $l=1,2$.

(ii) $|V_*(FG_1)|=\left\{ \begin{aligned}
           &|F|^{\frac{|G_1|+|\Omega_1(G_1)|}{2}-1}, \ \mathrm{if}\ 1\leq n_1\leq 2,\ \ m_1=1.\\
          &4|F|^{\frac{|G_1|+|\Omega_1(G_1)|}{2}-1}, \ \mathrm{if}\  n_1\geq 3, \  m_1\geq 2.\\
          &2|F|^{\frac{|G_1|+|\Omega_1(G_1)|}{2}-1}, \ \mathrm{otherwise}.
                                           \end{aligned} \right.$

       $|V_*(FG_2)|=\left\{ \begin{aligned}
           &8|F|^{\frac{|G_2|+|\Omega_1(G_2)|}{2}-1}, \ \mathrm{if}\  n_1=k=1,\  m_1\geq 2.\\
          &4|F|^{\frac{|G_2|+|\Omega_1(G_2)|}{2}-1}, \ \begin{aligned}
           &\ \mathrm{if} \ n_1=m_1=k=1\\
          &\mathrm{or}\ n_1=2, m_1\geq 2, \  k=1\\
          &\mathrm{or}\  n_1\geq 3,\ m_1\geq 2.
                                           \end{aligned}\\
          &|F|^{\frac{|G_2|+|\Omega_1(G_2)|}{2}-1}, \ \mathrm{if}\  1\leq n_1\leq 2, m_1=1,\ k\geq 2.\\
          &2|F|^{\frac{|G_2|+|\Omega_1(G_2)|}{2}-1}, \ \mathrm{otherwise}.
                                           \end{aligned} \right.$
\end{theorem}

\begin {proof}
 Let $G_l=\langle x_1 ,y_1\rangle\Y\langle x_2,y_2\rangle\Y\cdots\Y\langle x_k,y_k\rangle\Y\langle z_1\rangle\times\langle z_2\rangle$,
where $k\geq 1$, $\langle x_1 ,y_1 \rangle=H_l$,
$\langle x_i ,y_i \rangle\cong D_8, i=2,\ldots,k$, $\langle z_1\rangle\cong\mathbb{Z}_{2^{n_1}}$ and
$\langle z_2\rangle\cong\mathbb{Z}_{2^{m_1}}$.

(i) For any $g\in G$, we may let
$g=g_1g_2$, where $g_1=z_1^{i}z_2^{j}$, $0\leq i< 2^{n_1-1}$, $0\leq j< 2^{m_1}$, and $g_2\in \langle x_1,y_1,\ldots,x_k,y_k\rangle$.
If $g_1^2=1$, then $z_1^{2i}z_2^{2j}=1$. Hence $i=0$, $j=0$ or $2^{m_1-1}$.
If $g_1^2=c$, then $n_1\geq 2$ and $z_1^{2i}z_2^{2j}=z_1^{2^{n_1-1}}$. Hence $i=2^{n_1-2}$, $j=0$ or $2^{m_1-1}$.

(1) If $n_1=1$, then
we have  $\Omega_1(G_1)=4\gamma_1(k)=2^{2k+1}+2^{k+1}=\Omega_c(G_2)$
and  $\Omega_c(G_1)=4\gamma_2(k)=2^{2k+1}-2^{k+1}=\Omega_1(G_2)$.

(2) If $n_1\geq 2$,  then
we have  $\Omega_1(G_l)=4\gamma_1(k)+4\gamma_2(k)=2^{2k+2}=\Omega_c(G_l)$.

(ii) Suppose  $n_1=1$ and $G_2=Q_8\times\langle z_2\rangle$. Note that
the group $\langle 1+\sum\limits_{g \in \Omega_c(G_2)}\alpha_gg\widehat{G_2'}\in S_{G_2'}, \alpha_g\in F\rangle$
is generated by the following elements
$$1+(\alpha^2+\alpha)x_1\widehat{G_2'},1+(\alpha^2+\alpha)y_1\widehat{G_2'},1+\alpha(x_1+y_1+x_1y_1)\widehat{G_2'},$$ $$1+\alpha(x_1+x_1z_2^{2^{m_1-1}})\widehat{G_2'},1+\alpha(y_1+y_1z_2^{2^{m_1-1}})\widehat{G_2'},1+\alpha(x_1y_1+x_1y_1z_2^{2^{m_1-1}})\widehat{G_2'},$$
where $\alpha\in F$,
according to (i) of Lemma \ref{INS} and (ii) of Lemma \ref{GINS}. Hence
$\Theta(G_2)=\frac{1}{4}|F|^{6}=\frac{1}{4}|F|^{\frac{\Omega_c(G_2)}{2}}.$
At this time, by Lemma \ref{S2}, $|V_*(FG_2)|=4|F|^{\frac{|G_2|-2^{3}}{4}}\cdot|V_*(F\overline{G}_2)|.$
Note that $\overline{G}_2\cong \mathbb{Z}_{2}\times\mathbb{Z}_{2}\times\mathbb{Z}_{2^{m_1}}$.
Hence
$$|V_*(FG_2)|=\left\{ \begin{aligned}
          &4|F|^{2^{m_1+2}+1}=4|F|^{\frac{|G_2|+|\Omega_1(G_2)|}{2}-1}, \ \mathrm{if}\   m_1=1,\\
          &8|F|^{2^{m_1+2}+1}=8|F|^{\frac{|G_2|+|\Omega_1(G_2)|}{2}-1}, \ \mathrm{if}\   m_1\geq 2.
                                          \end{aligned} \right.$$

For the other cases when $n_1=1$,  $D_8$ will appear in the central product of $G_l$.
At this time, we have $1+\alpha g\widehat{G_l'}\in S_{G_l'}$ for any $\alpha\in F$ and $g\in \Omega_c(G_l)$ by Lemma \ref{USC}.
From this,
by Lemma \ref{S2},  $$|V_*(FG_l)|=|F|^{\frac{|G_l|+|\Omega_1(G_l)|-|\Omega_c(G_l)|}{4}}\cdot|V_*(F\overline{G}_l)|.$$
Note that $\overline{G}_l\cong \mathbb{Z}_{2}^{(2k)}\times\mathbb{Z}_{2^{m_1}}$. Hence
$$|V_*(FG_l)|=\left\{ \begin{aligned}
          &|F|^{2^{2k+m_1}+2^{2k}+2^k\epsilon_l-1}=|F|^{\frac{|G_l|+|\Omega_1(G_l)|}{2}-1}, \ \mathrm{if}\   m_1=1,\\
          &2|F|^{2^{2k+m_1}+2^{2k}+2^k\epsilon_l-1}=2|F|^{\frac{|G_l|+|\Omega_1(G_l)|}{2}-1}, \ \mathrm{if}\   m_1\geq 2,
                                          \end{aligned} \right. $$
where $\epsilon_1=1$ and $\epsilon_2=-1$.

Suppose that $n_1=2$ and $G_2\cong Q_8\Y\mathbb{Z}_4\times\mathbb{Z}_{2^{m_1}}$. Similar to (ii) of Lemma \ref{M11DQ}, we  $\Theta(G_2)=\frac{1}{2}|F|^{\frac{\Omega_c(G_2)}{2}}$.
By Lemma \ref{S2},  $|V_*(FG_2)|=2|F|^{\frac{|G_2|}{4}}\cdot|V_*(F\overline{G}_2)|.$
Note that $\overline{G}_2\cong \mathbb{Z}_{2}^{(3)}\times\mathbb{Z}_{2^{m_1}}$. Hence
$$|V_*(FG_2)|=\left\{ \begin{aligned}
          &2|F|^{2^{m_1+3}+7}=2|F|^{\frac{|G_2|+|\Omega_1(G_2)|}{2}-1}, \ \mathrm{if}\   m_1=1,\\
          &4|F|^{2^{m_1+3}+7}=4|F|^{\frac{|G_2|+|\Omega_1(G_2)|}{2}-1}, \ \mathrm{if}\   m_1\geq 2.
                                          \end{aligned} \right.$$

For the other cases when $n_1=2$, we have $1+\alpha g\widehat{G_l'}\in S_{G_l'}$ for any $\alpha\in F$ and $g\in \Omega_c(G_l)$ by Lemma \ref{USC}.
From this,
by Lemma \ref{S2},  $|V_*(FG_l)|=|F|^{\frac{|G_l|}{4}}\cdot|V_*(F\overline{G}_l)|.$
Note that $\overline{G}_l\cong \mathbb{Z}_{2}^{(2k)}\times\mathbb{Z}_2\times\mathbb{Z}_{2^{m_1}}$. Hence
$$|V_*(FG_l)|=\left\{ \begin{aligned}
          &|F|^{2^{2k+m_1+1}+2^{2k+1}-1}=|F|^{\frac{|G_l|+|\Omega_1(G_l)|}{2}-1}, \ \mathrm{if}\    m_1=1,\\
          &2|F|^{2^{2k+m_1+1}+2^{2k+1}-1}=2|F|^{\frac{|G_l|+|\Omega_1(G_l)|}{2}-1}, \ \mathrm{if}\   m_1\geq2.
                                          \end{aligned} \right.$$

Suppose that $n_1\geq 3$, we have $1+\alpha g\widehat{G_l'}\in S_{G_l'}$ for any $\alpha\in F$ and $g\in \Omega_c(G_l)$ by Lemma \ref{USC}.
From this,
we have
$$|V_*(FG_l)|=\left\{ \begin{aligned}
          &2|F|^{2^{2k+n_1+m_1-1}+2^{2k+1}-1}=2|F|^{\frac{|G_l|+|\Omega_1(G_l)|}{2}-1}, \ \mathrm{if}\    m_1=1,\\
          &4|F|^{2^{2k+n_1+m_1-1}+2^{2k+1}-1}=4|F|^{\frac{|G_l|+|\Omega_1(G_l)|}{2}-1}, \ \mathrm{if}\   m_1\geq 2.
                                          \end{aligned} \right.$$

\end{proof}
Obviously, by Lemma \ref{UDP}, the unitary subgroups of  (ix) and (x) of Theorems \ref{ST1} and \ref{ST2}, (iii) and (iv) of Theorems \ref{ST3} and \ref{ST4}, Corollaries \ref{ST5} and \ref{ST6},  Theorem \ref{ST7} are the special cases of Theorem \ref{DQZZ}.

\begin{corollary}
Let $G$ be  a nonabelian  $2$-group given by a
central extension of the form
$$1\longrightarrow \mathbb{Z}_{2^n}\times \mathbb{Z}_{2^m} \longrightarrow G \longrightarrow \mathbb{Z}_2\times
\cdots\times \mathbb{Z}_2 \longrightarrow 1$$ and $G'\cong
\mathbb{Z}_2$,  $n\geq m\geq 1$. Then the order of $V_*(FG)$ can be divisible by $|F|^{\frac{1}{2}(|G|+|\Omega_1(G)|)-1}$.
\end{corollary}

\begin{proof}
 According to Lemmas \ref{UDP} and \ref{SCG}, it is convenient not to consider an elementary abelian $2$-group as a direct product term of $G$.
 For the types of Theorem \ref{ST1}, the result is true by Theorems \ref{USMD}, \ref{USDQ}, \ref{UMMD}, \ref{UM1D}, \ref{M11DQ} and \ref{DQZZ}.
 For the types of Theorem \ref{ST2}, the result is true by Theorems \ref{USMD2}, \ref{UM1D2},  \ref{DQZZ} and  Corollaries \ref{UMMD2} and \ref{M11DQ2}.
 For the types of Theorem \ref{ST3}, the result is true by Theorems \ref{M11DQ} and \ref{DQZZ}.
For the types of Theorem \ref{ST4}, the result is true by Theorems \ref{UM1D2} and \ref{DQZZ}.
For the types of Corollary \ref{ST5}, the result is true by Corollary \ref{UM1DC} and Theorem \ref{DQZZ}.
For the types of Corollary \ref{ST6}, the result is true by Corollary \ref{M11DQ3} and Theorem \ref{DQZZ}.
For the types of Theorem \ref{ST7}, the result is true by  Theorem \ref{DQZZ}.

\end{proof}

%%%%%%%%%%%%%%%%%%%%%%%%%%%%%%%%%%%%%%%%%%%%%%%%%%%%%%%%%%%%%%%%%%%%%%

%%%%%%%%%%%%%%%%%%%%%%%%%%%%%%%%%%%%%%%%%%%%%%%%%%%%%%%
%%% Acknowledgements. ÖÂÐ»
%%%%%%%%%%%%%%%%%%%%%%%%%%%%%%%%%%%%%%%%%%%%%%%%%%%%%%%
\Acknowledgements{This work was supported by National Natural Science Foundation of China (Grant No. 12171142). We cordially
thank the referees for their time and helpful comments.}

%%%%%%%%%%%%%%%%%%%%%%%%%%%%%%%%%%%%%%%%%%%%%%%%%%%%%%%
%%% Conflict of interest. ×÷ÕßÀûÒæÉùÃ÷
%%%%%%%%%%%%%%%%%%%%%%%%%%%%%%%%%%%%%%%%%%%%%%%%%%%%%%%
%\InterestConflict

%%%%%%%%%%%%%%%%%%%%%%%%%%%%%%%%%%%%%%%%%%%%%%%%%%%%%%%
%%% Supplements. ²¹³ä²ÄÁÏ, ·Ç±ØÑ¡
%%%%%%%%%%%%%%%%%%%%%%%%%%%%%%%%%%%%%%%%%%%%%%%%%%%%%%%
%\Supplements{}

%%%%%%%%%%%%%%%%%%%%%%%%%%%%%%%%%%%%%%%%%%%%%%%%%%%%%%%
%%% Reference section. ²Î¿ŒÎÄÏ×
%%% citation in the content using "some words~\cite{1,2}".
%%% ~ is needed to make the reference number is on the same line with the word before it.
%%%%%%%%%%%%%%%%%%%%%%%%%%%%%%%%%%%%%%%%%%%%%%%%%%%%%%%

\end{document}